\documentclass[11pt]{article}

\usepackage{amsthm}
\usepackage{amssymb}
\usepackage{tikz}
\usepackage{verbatim}
\usepackage[numbers]{natbib}

\usepackage{rotating}
\usepackage{url}
\usepackage{float}
\usepackage{graphicx}
\usepackage{color}
\usepackage{latexsym}
\usepackage{amsfonts}
\usepackage{amsmath}
\usepackage{psfrag}
\usepackage{mathrsfs}
\usepackage{algpseudocode}
\usepackage[ruled]{algorithm2e}
\usepackage{caption}
\usepackage{multirow}

\usepackage{setspace}

\usepackage[textsize=tiny]{todonotes} 

\newtheorem{theorem}{Theorem}

\newtheorem{lemma}{Lemma}
\newtheorem{example}{Example}

\newtheorem{remark}{Remark}

\DeclareMathOperator*{\argmin}{argmin} 

\DeclareMathOperator*{\aff}{aff}
\DeclareMathOperator*{\conv}{conv}
\DeclareMathOperator*{\epi}{epi}

\usepackage{mathrsfs}

\def\R{\mathbb{R}}

 \setbox\strutbox=\hbox{\vrule height7pt depth2pt width0pt}

\def\square{{\setbox0=\hbox{X}\hbox to \ht0{\vrule\hss\vbox to \ht0{
  \hrule width \ht0\vfil\hrule width \ht0}\vrule}}}

\title{An oracle-based framework for robust combinatorial optimization}

\author{Enrico Bettiol\thanks{ TU Dortmund University \{enrico.bettiol@math.tu-dortmund.de\}},
Christoph Buchheim\thanks{ TU Dortmund University \{christoph.buchheim@math.tu-dortmund.de\}}, 
Marianna De Santis\thanks{ Sapienza University of Rome \{marianna.desantis@uniroma1.it\}}, 
Francesco Rinaldi\thanks{ University of Padova \{rinaldi@math.unipd.it\}}}

\begin{document}

\maketitle

\begin{abstract}
  We propose a general solution approach for min-max-robust
  counterparts of combinatorial optimization problems with uncertain
  linear objectives. We focus on the discrete scenario case, but our
  approach can be extended to other types of uncertainty sets such as
  polytopes or ellipsoids. Concerning the underlying certain problem,
  the algorithm is entirely oracle-based, i.e., our approach only
  requires a (primal) algorithm for solving the certain problem. It is
  thus particularly useful in case the underlying problem is hard to
  solve, or only defined implicitly by a given software addressing the
  certain case. The idea of our algorithm is to solve the convex
  relaxation of the robust problem by a simplicial decomposition
  approach, the main challenge being the non-differentiability of the
  objective function in the case of discrete or polytopal
  uncertainty. The resulting dual bounds are then used within a
  tailored branch-and-bound framework for solving the robust problem
  to optimality. By a computational evaluation, we show that our
  method outperforms straightforward linearization approaches on the robust
  minimum spanning tree problem. Moreover, using the Concorde solver
  for the certain oracle, our approach computes much better dual
  bounds for the robust traveling salesman problem in the same amount
  of time.
\end{abstract}


\pagestyle{plain} \setcounter{page}{1}


\section{Introduction}

Robust optimization has become a wide and active research area in the
last decades. The aim is to address optimization problems with
uncertain data. Unlike the stochastic optimization problem, which
usually aims at optimizing expected values, the robust optimization
paradigm tries to optimize the worst case. While stochastic
optimization requires full knowledge of the probability distributions
of all uncertain problem data, robust optimization only asks for a
so-called uncertainty sets containing all scenarios that need to be
taken into account. While generally leading to computationally easier
problems than stochastic optimization, it is well-known that robust
counterparts of tractable combinatorial optimization problems usually
turn out to be NP-hard for most types of uncertainty sets; see,
e.g.,~\cite{kouvelis} or the recent survey~\cite{bksurvey} and the
references therein.

In this paper, we address robust counterparts of general combinatorial
optimization problems of the type
\begin{equation}\label{CertainProb}\tag{P}
      \begin{array}{l l}
        \min & c^\top x+c_0\\[1ex]
        \textnormal{\,s.t. }& x\in X,
      \end{array}
\end{equation}
where~$X\subseteq\{0,1\}^n$ is any set of binary vectors describing
the feasible solutions of the problem at hand. The objective function
coefficients~$(c_0,c)\in\R^{n+1}$ are considered uncertain. The robust
counterpart of \eqref{CertainProb} is then given by
\begin{equation}\label{MinMaxProb}\tag{R}
      \begin{array}{l l}
        \min & \max_{(c_0,c)\in U}\;c^\top x+c_0\\[1ex]
        \textnormal{\,s.t. }& x\in X,
      \end{array}
\end{equation}
where~$U\subseteq\R^{n+1}$ is the so-called \emph{uncertainty set},
collecting all likely scenarios. Note that allowing an uncertain
constant~$c_0$ makes the approach slightly more general, even though
the latter is not relevant in the deterministic
problem~\eqref{CertainProb}. With respect to the considered type of
uncertainty set, our approach is rather general, but we will
concentrate our exposition on the so-called \emph{discrete
  uncertainty} case, where~$U$ is given as a finite
set. Other classes of uncertainty sets often considered in the
literature include polytopal or ellipsoidal sets.

While many approaches devised in the literature consider special
classes of combinatorial structures~$X$, our aim is to devise an
entirely oracle-based approach. We thus assume that we have at our
disposition an algorithm that solves Problem~\eqref{CertainProb}, for
any given objective~$c$, but we do not pose any restrictions on how
this algorithm works. Our approach is thus particularly well-suited in
situations where the certain problem is already NP-hard but
well-studied, such as, e.g., the traveling salesman problem, or where
the underlying problem is not a classical textbook optimization
problem, but given by some sophisticated and probably obscure solution
software. Our approach does not require any knowledge about the
underlying problem.
  
As mentioned above, robust counterparts are often NP-hard even in
cases where the underlying problem~\eqref{CertainProb} is
tractable. Consequently, in order to solve~\eqref{MinMaxProb}, it
cannot suffice to call the oracle a polynomial number of times. This
is even true without assuming
P\;$\neq$\;NP~\cite{buchheim20}. Instead, we propose a
branch-and-bound approach, where the main ingredient is the
computation of the lower bound given by the straightfoward convex
relaxation of~\eqref{MinMaxProb}, namely
\begin{equation}\label{CR}\tag{C}
  \begin{array}{l l}
    \min & \max_{(c_0,c)\in U}\;c^\top x+c_0\\[1ex]
    \textnormal{ s.t. }& x\in \conv(X).
  \end{array}
\end{equation}
This problem is well-defined and convex, as long as~$U$ is any compact
set. While ellipsoidal uncertainty leads to a smooth objective
in~\eqref{CR}, which can be exploited
algorithmically~\cite{bdrt17,bd18}, the discrete and the polytopal
uncertainty cases lead to piecewise linear objective functions, requiring
different solution methods.

In our approach, Problem~\eqref{CR} is solved by an \emph{inner
  approximation algorithm}; see, e.g.,~\cite{Ber2015} and the
references therein. It belongs to the class of \emph{Simplicial
  Decomposition} (SD) methods.  First introduced by Holloway in
\cite{holloway1974extension} and then further studied in
\cite{hearn1987restricted,larsson1992simplicial,ventura1993restricted,von1977simplicial},
SD methods currently represent a standard tool in convex
optimization. Our SD method makes use of two different oracles: the
first one is an algorithm for solving the convex relaxation over an
inner approximation of $\conv(X)$, being the convex hull of a
subset~$X'$ of~$X$. It is important to notice that such a subroutine
implicitly defines the uncertainty set~$U$, while the rest of our
algorithm is independent of~$U$. The second oracle is the one
described above, which implicitly defines the set~$X$ and hence
also~$\conv(X)$.
Our approach can thus be seen as an
oracle-based version of a generalized SD algorithm; see, e.g.,
\cite{Ber2015, BertsekasY11} for further details about generalized
SD. The proposed method indeed performs a two-step optimization
process by handling an ever expanding inner approximation of the
relaxed feasible set $\conv(X)$.
At a given iteration, the method
first builds up a reduced problem (whose feasible set is given by the
inner approximation) and solves it by means of the first oracle. It
then feeds the second oracle with the information coming from the
first step to hopefully generate new extreme points that guarantee a
refinement of the inner approximation. If a new point cannot be found,
then the solution obtained with the last reduced problem is the
optimal one. The way the refinement step is carried out is crucial to
guarantee finite convergence of our method in the end.

\emph{Dropping rules} (i.e., rules that allow to get rid of useless
points in the inner approximation) are often used in simplicial
decomposition like algorithms to keep the computational cost deriving
from the first oracle small enough; see, e.g., \cite{Ber2015,
  bettiol2020conjugate,von1977simplicial}. As pointed out in
\cite{BertsekasY11}, defining suitable dropping rules for a
generalized simplicial decomposition, while guaranteeing finite
convergence of the method, is a challenging task. We propose a simple
dropping rule and analyze it in depth both from a theoretical and a
computational point of view.

Some other oracle-based algorithms for robust combinatorial
optimization with objective function uncertainty have been devised in
the literature. In particular, tailored column generation approaches
for dealing with the continuous relaxation of the given combinatorial
problem are studied in \cite{buchheim2017min, kammerling2020oracle}.
When considering Problem~\eqref{CR}, those column generation
algorithms turn out to be closely related to a Kelly's cutting plane
approach for the problem
$$\max_{(c_0,c)\in U}\;\min_{x\in \conv(X)}\;c^\top x+c_0\;,$$ which
is equivalent to~\eqref{CR} in case of convex~$U$ by the minimax
theorem. Another interesting approach to handle the
relaxation~\eqref{CR} is described in \cite{kurtz2021new}, where the
author proposes a projected subgradient method that approximately
solves the projection problem at each iteration by the classical
Frank-Wolfe algorithm. This approach is somehow related to
gradient-sliding methods, see, e.g., \cite{lan2020first} and the
references therein, and hence obviously differs from the one described
in this paper.

When aiming at general approaches that do no not exploit specific
characteristics of the underlying problem~\eqref{CertainProb}, the
main alternative to oracle-based algorithms are approaches based on an
IP-formulation of~\eqref{CertainProb}. For discrete uncertainty, the
non-linear objective in~\eqref{CR} can easily be linearized, and this
approach can be extended to infinite uncertainty sets~$U$ using a
dynamic generation of worst-case scenarios, provided that a linear
optimization oracle over~$U$ is given; see~\cite{mutapcic09} for a general analysis and~\cite{fischetti12} for an experimental 
comparison with reformulation-based approaches. 
The scenario generation method is still
applicable when having only a separation algorithm for~$\conv(X)$ at
hand. In the experimental evaluation presented in this paper, we
compare our SD approach to such a separation oracle based approach for
the discrete uncertainty case, using CPLEX to solve the resulting
integer linear problems.

In the subsequent section, we describe our SD approach in more detail,
concentrating on the discrete uncertainty case and with a particular
focus on dropping rules. In Section~\ref{sec:Embed_SD}, we explain how
we embedded this approach into a branch-and-bound framework. An
experimental evaluation is presented in
Section~\ref{section:results}. Section~\ref{section:conclusion}
concludes.

\section{Computation of lower bounds}  

The main ingredient in our approach is the computation of the lower
bound given by the convex relaxation of the robust
counterpart~\eqref{MinMaxProb}. Setting~$P:=\conv(X)$ and $f(x) :=
\max_{(c_0,c)\in U} c^\top x+c_0$, the problem we address is thus given as
\begin{equation}\label{ContRel0}\tag{CR}
  \begin{array}{l l}
    \min & f(x)\\[1ex]
    \textnormal{ s.t. }& x\in P.
  \end{array}
\end{equation}
It is easy to see that the objective function~$f$ in~\eqref{ContRel0}
is convex for any uncertainty set~$U$, however, it is not necessarily
differentiable. E.g., in case of a finite set~$U$, differentiability
is guaranteed only in points~$\bar x$ where the scenario~$(c_0,c)\in
U$ maximizing~$c^\top \bar x+c_0$ is unique. In the following, we
first describe the general idea of the simplicial decomposition
approach applied to the potentially non-differentiable
problem~\eqref{ContRel0}; see Section~\ref{sec:general}. Afterwards,
we investigate a variant of the approach where vertices are dropped
in case they are not needed to define the current simplex. This
however requires to deal with the issue of cycling; see
Section~\ref{sec:eliminate}.

\subsection{General approach}\label{sec:general}

We now describe the two oracles that we embed in our SD framework. The first oracle SIM-O essentially minimizes~$f$ over a simplex given
by the convex hull of a finite set~$V\subset\R^n$. Beyond the optimal solution~$x^*$, we
also need coefficients yielding~$x^*$ as a convex combination of
points in~$V$ and a subgradient~$c$ of~$f$ in~$x^*$ such that~$-c$
belongs to the normal cone of $\conv(V)$ in~$x^*$. The existence of
such~$c$ is a necessary and sufficient condition of optimality
for~$x^*$.
\begin{algorithm}\caption{SIM-O}\label{oracle2}
    \SetKwInOut{Input}{input}\SetKwInOut{Output}{output}
      \BlankLine
      \Input{finite subset $V\subset \R^n$}
      \BlankLine
      \Output{$ \alpha^* \in \R^{V}_+$ with $\sum_{v\in V}\alpha^*_v = 1$,\\ 
      $x^* = \sum_{v\in V} \alpha^*_v v$, and\\
      $c^* \in \partial f(x^*)\cap(-\mathcal{N}_{\conv(V)}(x^*))$}
      \BlankLine
\end{algorithm}

The second oracle, namely Oracle~LIN-O, is the main oracle defining the underlying
problem. It takes as input an objective vector~$c$ and returns a
minimizer of~$\min_{x\in X}c^\top x$, which is the same as solving
Problem~\eqref{CertainProb}.

\begin{algorithm}\caption{LIN-O}\label{oracle1}
    \SetKwInOut{Input}{input}\SetKwInOut{Output}{output}
      \BlankLine
      \Input{$c\in \R^n$}
      \BlankLine
      \Output{optimizer $x^*$ of $\min_{x\in X}\;c^\top x$}
      \BlankLine
\end{algorithm}

Using these oracles, Algorithm~\texttt{SD} works as follows (see the
pseudo-code below): the set~$V^k$ is initialized as the singleton
$\{\hat x^0\}$, where~$\hat x^0$ is an arbitrary element of~$X$.  Then, we enter
a loop.  At each iteration $k$, oracle SIM-O is first called, in order
to calculate a minimizer~$x^k$ of~$f$ over~$\conv(V^k)$ and a
subgradient
$$c^k\in \partial f(x^k)\cap(-\mathcal{N}_{\conv(V^k)}(x^k))\;.$$
Then, oracle LIN-O is called, giving as output a minimizer~$\hat x^k$
of~$(c^k)^\top x$ over~$x\in X$.  Note that both~$x^k$ and $\hat x^k$
belong to $P=\conv (X)$, but not necessarily to~$X$. From the
definition of~$\mathcal{N}_{\conv(V^k)}(x^k)$ we have that
 \[(c^k)^\top x \geq (c^k)^\top x^k\quad \forall\; x\in \conv(V^k).\]
This means that as long as $(c^k)^\top \hat x^k < (c^k)^\top x^k$ we
can go further in the minimization of $f$ over $P$ by including the
point $\hat x^k$ in the set $V^k$.  Otherwise, if ${c^k}^\top \hat x^k
\geq {c^k}^\top x^k$ we can stop our algorithm, as $x^k$ is a
minimizer of $f$ over $P$ and $f(x^k)$ is a lower bound for
Problem~\eqref{MinMaxProb}. See Fig.~\ref{fig:sd} for an illustration.
\begin{algorithm}\caption{\texttt{SD}}\label{fig:SD}
  \BlankLine
  \SetKwInOut{Given}{given}\SetKwInOut{Output}{output}
  \Given{oracles LIN-O and SIM-O}
  \BlankLine
  \Output{optimizer $x^*$ of~\eqref{ContRel0}}
  \BlankLine
  compute any $\hat x^0\in X$ by calling LIN-O with arbitrary objective\\
  set $V^1 = \{\hat x^0\}$\\
  \For{$k=1,2,\ldots$}{
    compute  $\alpha^k, x^k, c^k$ by calling SIM-O for the set $V^k$\\
    compute $\hat x^k$ by calling LIN-O with objective $c^k$\\ 
    \If{$(c^k)^\top\hat x^k \geq (c^k)^\top x^k$}{STOP: $x^k$ minimizes $f$ over $P$}
    set $V^{k+1} := V^k \cup \{\hat x^k\}$}
\end{algorithm}

\begin{figure}
  \begin{center}
    \begin{tikzpicture}[scale=2]
      \draw[->, thick, dashed, color=red] (0,0) -- (0.15,0.4) node[left] {$-c^1$};
      \node[scale=0.5, draw, circle, fill=white] at (0,0) {};
      \node[scale=0.5, draw, circle, fill=white] at (0,1) {};
      \node[scale=0.5, draw, circle, fill=white] at (1,0) {};
      \node[scale=0.5, draw, circle, fill=white] at (1,1) {};
      \node[scale=0.5, draw, circle, fill=black] at (0,0) {};
      \node[color=black] at (-0.15,0) {$\hat x^0$};
    \end{tikzpicture}
    \qquad\qquad
    \begin{tikzpicture}[scale=2]
      \draw[->, thick, dashed, color=red] (0.4,0.4) -- (0.18,0.62) node[left] {$-c^2$};
      \node[color=red,right] at (0.4,0.4) {$x^1$};
      \node[scale=0.5, draw, circle, fill=white] at (0,0) {};
      \node[scale=0.5, draw, circle, fill=white] at (0,1) {};
      \node[scale=0.5, draw, circle, fill=white] at (1,0) {};
      \node[scale=0.5, draw, circle, fill=white] at (1,1) {};
      \node[scale=0.5, draw, circle, fill=black] at (0,0) {};
      \node[color=black] at (-0.15,0) {$\hat x^0$};
      \node[scale=0.5, draw, circle, fill=black] at (1,1) {};
      \node[color=black, scale=0.5, draw, circle, fill=black] at (1,1) {};
      \node[color=black] at (1.15,1) {$\hat x^1$};
      \draw[thick] (0,0) -- (1,1);
      \node[scale=0.5, draw, circle, color=red, fill=red] at (0.4,0.4) {};
    \end{tikzpicture}
    \quad
    \begin{tikzpicture}[scale=2]
      \draw[thick, fill=gray!30] (0,0) -- (1,1) -- (0,1) -- cycle;
      \draw[->, thick, dashed, color=red] (0,0.8) -- (-0.4,0.8) node[left] {$-c^3$};
      \node[color=red,right] at (0,0.8) {$x^2$};
      \draw[->, thick, dashed, color=white] (1,0.8) -- (1.8,0.8);
      \node[scale=0.5, draw, circle, fill=white] at (0,0) {};
      \node[scale=0.5, draw, circle, fill=white] at (0,1) {};
      \node[scale=0.5, draw, circle, fill=white] at (1,0) {};
      \node[scale=0.5, draw, circle, fill=white] at (1,1) {};
      \node[scale=0.5, draw, circle, fill=black] at (0,0) {};
      \node[color=black] at (-0.15,0) {$\hat x^0$};
      \node[scale=0.5, draw, circle, fill=black] at (1,1) {};
      \node[color=black] at (1.15,1) {$\hat x^1$};
      \draw[thick] (0,0) -- (1,1);
      \node[color=black] at (-0.15,1) {$\hat x^2$};
      \node[scale=0.5, draw, circle, fill=black] at (0,1) {};
      \node[scale=0.5, draw, circle, color=red, fill=red] at (0,0.8) {};
    \end{tikzpicture}
  \end{center}
  \caption{Illustration of Algorithm~\texttt{SD} with~$X=\{0,1\}^2$.}\label{fig:sd}
\end{figure}
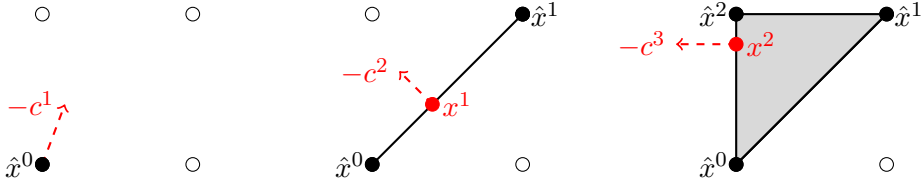

We claim that Algorithm~\texttt{SD} terminates after finitely many iterations
with a correct result. For showing this, first observe
\begin{lemma}\label{lemma_lb}
  At every iteration~$k$ of Algorithm~{\texttt{SD}}, a lower bound for
  Problem~\eqref{ContRel0} is given by~$f(x^k)+(c^k)^\top (\hat x^k-x^k)$.
\end{lemma}
\begin{proof}
  Define~$c_0:=f(x^k)-(c^k)^\top x^k$. Since~$c^k\in \partial f(x^k)$,
  and by the choice of~$\hat x^k$, we obtain
  \[
  f(\bar x)
  \geq f(x^k) + (c^k)^\top (\bar x - x^k)
  = c_0 + (c^k)^\top \bar x
  \geq c_0 + \min_{x\in P}\;(c^k)^\top x
  = c_0 + (c^k)^\top \hat x^k
  \]
  for all $\bar x\in P$, so that $c_0 + (c^k)^\top \hat
  x^k=f(x^k)+(c^k)^\top (\hat x^k-x^k)$ is a lower bound for
  Problem~\eqref{ContRel0}.
\end{proof}
\begin{theorem}
  Algorithm~{\texttt{SD}} terminates after a finite number of iterations with a
  correct result.
 \end{theorem}
\begin{proof}
  Correctness immediately follows from Lemma~\ref{lemma_lb},
  since~$f(x^k)$ is clearly an upper bound for
  Problem~\eqref{ContRel0} and the algorithm only terminates
  when~$(c^k)^\top \hat x^k=(c^k)^\top x^k$.  So it remains to show
  finiteness.  From the definition of $\mathcal{N}_{\conv(V^k)}(x^k)$
  we have that
  \[(c^k)^\top x \geq (c^k)^\top x^k\quad \forall\; x\in \conv(V^k).\]
  This means that in case Algorithm~{\texttt{SD}} does not terminate at
  iteration~$k$, the point $\hat x^k\in X$ does not belong to $V^k$, so
  that $V^{k+1}$ is a strict extension of $V^k$. The result then
  follows from the finiteness of~$X$.
\end{proof}
Note that this proof of convergence relies on our general assumption that~$X$ is
a finite set and on the fact that we never eliminate vertices
of~$V^k$. The situation is more complicated when such an elimination
is allowed, as discussed in Section~\ref{sec:eliminate} below.

In the remainder of this subsection, we concentrate on the important
special case that~$U$ consists of a finite number of scenarios $\{c_1,
c_2, \ldots, c_m\}\subseteq\R^{n+1}$, where we denote~$c_i=(\tilde c_{i},\bar c_i)$ with the uncertain constant being~$\tilde c_{i}$. In this case, the oracle SIM-O can be realized as follows:
first note that we essentially need to solve
the problem
\begin{equation}\label{prob:6.8}
 \min_{x\in \conv(V^k)} f(x) = \min_{x\in \conv(V^k)} \max\{\bar c_1^\top x+\tilde c_{1},\, \bar c_2^\top x+\tilde c_{2},\ldots, \bar c_m^\top x+\tilde c_{m}\}.   
\end{equation}
In the following, we denote by $x^k$ the minimizer of~\eqref{prob:6.8}, adopting the same notation used within Algorithm~\texttt{SD}.
As mentioned in~\cite{BertsekasY11}, having a finite number of scenarios is one of the special cases where the calculation of a subgradient 
$c^k \in \partial f(x^k)\cap(-\mathcal{N}_{\conv(V^k)}(x^k))$ can be obtained as a byproduct of the solution of~\eqref{prob:6.8}.
For sake of completeness, we report how the subgradient $c^k$ is derived.
Problem~\eqref{prob:6.8} can be rewritten as
\begin{equation}\label{eq:probdiscrete}
\begin{array}{l l}
        \min & z\\[1ex]
        \textnormal{ s.t. }& \bar c_j^\top x +\tilde c_j \leq z,\quad j=1,\ldots,m\\[1.2ex]
        & x\in \conv(V^k)\;.
\end{array}
\end{equation}
From the optimality conditions of~\eqref{eq:probdiscrete}, we have
that the optimal solution $(x^k,z^k)$, together with the dual optimal
variables $\lambda^k_j$, satisfies

\[
\begin{array}{l r}
z^k = f(x^k) =  \max\{\bar c_1^\top x^k+\tilde c_1,\, \bar c_2^\top x^k+\tilde c_2,\ldots, \bar c_m^\top x^k+\tilde c_m\} \\[2.5ex]
x^k\in \conv(V^k),\;\quad  \bar c_j^\top x^k +\tilde c_j\leq z^k \qquad j=1,\ldots,m & \mbox{(primal feasibility)}\\[2.5ex]
\displaystyle(x^k, z^k) \in \argmin_{x\in\conv(V^k),\; z\in \R} \left\{\left( 1 - \textstyle\sum_{j=1}^m \lambda^k_j\right)z +  
\textstyle\sum_{j=1}^m \lambda^k_j \bar c_j^\top x\right\} & \mbox{(Lagrangian optimality)}\\[4.0ex]
\lambda_j^k \geq 0, & \mbox{(dual feasibility)}\\[3.5ex]
\lambda_j^k = 0 \quad \mbox{ if } \;\; \bar c_j^\top x^k +\tilde c_j< z^k = f(x^k)\qquad j=1,\ldots,m &\mbox{(complementary slackness)}
\end{array}
\]

It follows that $\sum_{j=1}^m \lambda^k_j = 1$ and, from Lagrangian
optimality, we have
\begin{equation}\label{eq:subgrad}
\Big(\sum_{j=1}^m \lambda^k_j \bar c_j \Big)^\top (x - x^k)\geq 0,\qquad \forall x\in \conv(V^k). 
\end{equation}
It can be shown (see~\cite{Ber2009}, p.199) that the vector $c^k :=
\sum_{j=1}^m \lambda^k_j \bar c_j$ is a subgradient of $f$ at $x^k$,
and~\eqref{eq:subgrad} implies that $-c^k$ belongs to the normal cone
of $\conv(V^k)$ at $x^k$, so that we indeed have~$c^k \in
\partial f(x^k)\cap(-\mathcal{N}_{conv(V^k)}(x^k))$.
Summarizing, when the set $U$ is finite, an oracle SIM-O suited for
our purposes can be implemented by any linear programming solver able to
address Problem~\eqref{eq:probdiscrete}, rewritten considering the
$\alpha_v^k$ as variables. In this way, $x^k$ is obtained as the
convex combination of $\alpha_v^k$.

In case of a differentiable function~$f$, the choice of~$c^k$ is
unique. In case of finite~$U$, one may ask the question whether there
is some freedom in the choice of~$c^k$, which could potentially be
exploited in order to find particularly promising search
directions. However, it turns out that even in the discrete
uncertainty case, the subgradient~$c^k$ is unique with high
probability when the scenarios are chosen (or perturbed) randomly.
\begin{theorem}\label{th:c_unique}
  Assume that all scenarios in~$U=\{c_1,\dots,c_m\}$ are perturbed by
  any continuously distributed random vector in~$\R^{m(n+1)}$ with
  full-dimensional support. Then, with probability one, the
  set~$\partial f(x^k)\cap(-\mathcal{N}_{\conv(V^k)}(x^k))$ is a
  singleton in each iteration.
\end{theorem}
\begin{proof}
  By definition, there exist $z^k$ and~$\alpha^k$ such that
  $(z^k,x^k,\alpha^k)$ is a basic optimal solution of
  \begin{equation}\label{eq:probdiscrete2}
    \begin{array}{l l}
      \min & z\\[1ex]
      \textnormal{ s.t. }& \bar c_j^\top x +\tilde c_j\leq z,\quad j=1,\ldots,m\\[1.2ex]
      & x=\sum_{v\in V^k} \alpha_v v\\
      & \alpha\ge 0\\
      & \sum_{v\in V^k} \alpha_v = 1\;.
    \end{array}
  \end{equation}
  Define~$A^=:=\{v\in V^k\mid \alpha^k_v=0\}$ and
  $C^=:=\{j\in\{1,\dots,m\}\mid \bar c_j^\top x^k+\tilde c_j=z^k\}$. As the feasible
  set of~\eqref{eq:probdiscrete2} has dimension~$|V^k|$, we have
  $|C^=|+|A^=|\ge |V^k|$, and equality holds with probability one.
  Now~$\partial f(x^k)=\conv\{\bar c_j\mid j\in C^=\}$ has dimension at
  most~$|C^=|-1$ and $\mathcal{N}_{\conv(V^k)}(x^k)$ has dimension
  $n-(|V^k|-|A^=|-1)$. Consequently, with probability one, the sum of
  the two dimensions is at most~$n$. Again with probability one, it
  follows that~$\partial f(x^k)$ and~$-\mathcal{N}_{\conv(V^k)}(x^k)$
  intersect in at most one point.
\end{proof}

\subsection{Vertex dropping rule}\label{sec:eliminate}

The running time of an iteration of Algorithm~\texttt{SD} strongly depends on
the size of~$V^k$. The overall performance could thus benefit from a
dropping rule for elements of~$V^k$. A straightforward idea is to
eliminate vertices not needed to define the minimizer of~$f$
over~$V^k$. We thus consider the following modified update rule:
\begin{equation}\label{eq:elrule}\tag{drop}
V^{k+1} := \{v\in V^k \mid \alpha^k_v >0\} \cup \{\hat x^k\}. 
\end{equation}
In the following, we will refer to Algorithm~\texttt{SD} where $V^k$ is updated
according to~\eqref{eq:elrule} as Algorithm~\texttt{SD-DROP}. In
case of a non-differentiable function~$f$, Algorithm~\texttt{SD-DROP} may
cycle, as shown in the following example.
\begin{example}\label{ex1}
 Let us consider the following problem  
\begin{equation*}
    \begin{array}{l l}
        \min & \max\{x_1 - x_2, x_2 - x_1\}\\[1ex]
        \textnormal{ s.t. } & x_1+ x_2 \leq1\\
        & x_1, x_2 \geq 0.
      \end{array}
\end{equation*}
Starting from $x^1={0\choose 0}$, Algorithm~{\texttt{SD-DROP}} will perform the
following iterations:
\begin{itemize}
 \item[k=1:] $x^1={0\choose 0}$, $V^1 = \{x^1\}$, $\alpha^1 = (1)$ and  $\partial f(x^1)\cap(-\mathcal{N}_{\conv(V^1)}(x^1))  
 =\conv\{ {1 \choose -1}, {-1 \choose 1}\}$.\\ We choose $c^1 = {1 \choose -1}$. Then $\hat x^1={0\choose 1}$, $V^2 = \{{0 \choose 0}, {0\choose 1}\}$
 \item[k=2:] $x^2={0\choose 0}$, $\alpha^2 = {1 \choose 0}$ and $\partial f(x^2)\cap(-\mathcal{N}_{\conv(V^2)}(x^2))  
 =\conv\{ {0 \choose 0}, {-1 \choose 1}\}$.\\ We choose $c^2 = {-1 \choose 1}$. Then $\hat x^2={1\choose 0}$, $V^3 = \{{0\choose 0}, {1\choose 0}\}$
\item[k=3:] $x^3={0\choose 0}$, $\alpha^3 = {1 \choose 0}$ and $\partial f(x^3)\cap(-\mathcal{N}_{\conv(V^3)}(x^3))
 =\conv\{ {0 \choose 0}, {1 \choose -1}\}$.\\ We choose $c^3 = {1 \choose -1}$. Then $\hat x^3={0\choose 1}$, $V^4 = \{{0\choose 0}, {0\choose 1}\}$.
 \end{itemize}
 At iteration $k=3$, we thus get $V^4 = V^2$ and the algorithm cycles. See Fig.~\ref{fig:ex1} for an illustration.
\end{example}
\begin{figure}
  \begin{center}
    \begin{tikzpicture}[scale=2]
      \draw[dotted] (0,2/3) -- (1/3,1);
      \draw[dotted] (0,1/3) -- (2/3,1);
      \draw[dotted,thick] (0,0) -- (1,1);
      \draw[dotted] (1/3,0) -- (1,2/3);
      \draw[dotted] (2/3,0) -- (1,1/3);
      \draw[->, thick, dashed, color=red] (0,0) -- (-0.3,0.3) node[left] {$-c^1$};
      \draw[->, thick, gray] (0.5-1/3,0.5+1/3) -- (0.5-1/3+1/10,0.5+1/3-1/10);
      \draw[->, thick, gray] (0.5-1/6,0.5+1/6) -- (0.5-1/6+1/10,0.5+1/6-1/10);
      \draw[->, thick, gray] (0.5+1/3,0.5-1/3) -- (0.5+1/3-1/10,0.5-1/3+1/10);
      \draw[->, thick, gray] (0.5+1/6,0.5-1/6) -- (0.5+1/6-1/10,0.5-1/6+1/10);
      \node[scale=0.5, draw, circle, fill=white] at (0,0) {};
      \node[scale=0.5, draw, circle, fill=white] at (0,1) {};
      \node[scale=0.5, draw, circle, fill=white] at (1,0) {};
      \node[scale=0.5, draw, circle, fill=black] at (0,0) {};
      \node[color=black] at (-0.15,0) {$\hat x^0$};
      \phantom{\draw[->, thick, dashed, color=red] (0,0) -- (0.3,-0.3) node[left] {$-c^2$};}
    \end{tikzpicture}
    \qquad
    \begin{tikzpicture}[scale=2]
      \draw[dotted] (0,2/3) -- (1/3,1);
      \draw[dotted] (0,1/3) -- (2/3,1);
      \draw[dotted,thick] (0,0) -- (1,1);
      \draw[dotted] (1/3,0) -- (1,2/3);
      \draw[dotted] (2/3,0) -- (1,1/3);
      \draw[->, thick, dashed, color=red] (0,0) -- (0.3,-0.3) node[left] {$-c^2$};
      \node[scale=0.5, draw, circle, fill=white] at (0,0) {};
      \node[scale=0.5, draw, circle, fill=black] at (0,1) {};
      \node[scale=0.5, draw, circle, fill=white] at (1,0) {};
      \node[scale=0.5, draw, circle, fill=black] at (0,0) {};
      \node[color=black] at (-0.15,0) {$\hat x^0$};
      \node[color=black] at (-0.15,1) {$\hat x^1$};
      \draw[thick] (0,0) -- (0,1);
    \end{tikzpicture}
    \qquad
    \begin{tikzpicture}[scale=2]
      \draw[dotted] (0,2/3) -- (1/3,1);
      \draw[dotted] (0,1/3) -- (2/3,1);
      \draw[dotted,thick] (0,0) -- (1,1);
      \draw[dotted] (1/3,0) -- (1,2/3);
      \draw[dotted] (2/3,0) -- (1,1/3);
      \draw[->, thick, dashed, color=red] (0,0) -- (-0.3,0.3) node[left] {$-c^3$};
      \node[scale=0.5, draw, circle, fill=white] at (0,0) {};
      \node[scale=0.5, draw, circle, fill=white] at (0,1) {};
      \node[scale=0.5, draw, circle, fill=black] at (1,0) {};
      \node[scale=0.5, draw, circle, fill=black] at (0,0) {};
      \node[color=black] at (-0.15,0) {$\hat x^0$};
      \node[color=black] at (1.15,0) {$\hat x^2$};
      \draw[thick] (0,0) -- (1,0);
      \phantom{\draw[->, thick, dashed, color=red] (0,0) -- (0.3,-0.3) node[left] {$-c^2$};}
    \end{tikzpicture}
    \qquad
    \begin{tikzpicture}[scale=2]
      \draw[dotted] (0,2/3) -- (1/3,1);
      \draw[dotted] (0,1/3) -- (2/3,1);
      \draw[dotted,thick] (0,0) -- (1,1);
      \draw[dotted] (1/3,0) -- (1,2/3);
      \draw[dotted] (2/3,0) -- (1,1/3);
      \draw[->, thick, dashed, color=red] (0,0) -- (0.3,-0.3) node[left] {$-c^4$};
      \node[scale=0.5, draw, circle, fill=white] at (0,0) {};
      \node[scale=0.5, draw, circle, fill=black] at (0,1) {};
      \node[scale=0.5, draw, circle, fill=white] at (1,0) {};
      \node[scale=0.5, draw, circle, fill=black] at (0,0) {};
      \node[color=black] at (-0.15,0) {$\hat x^0$};
      \node[color=black] at (-0.15,1) {$\hat x^3$};
      \draw[thick] (0,0) -- (0,1);
    \end{tikzpicture}
  \end{center}
  \caption{Illustration of Example~\ref{ex1}.}\label{fig:ex1}
\end{figure}
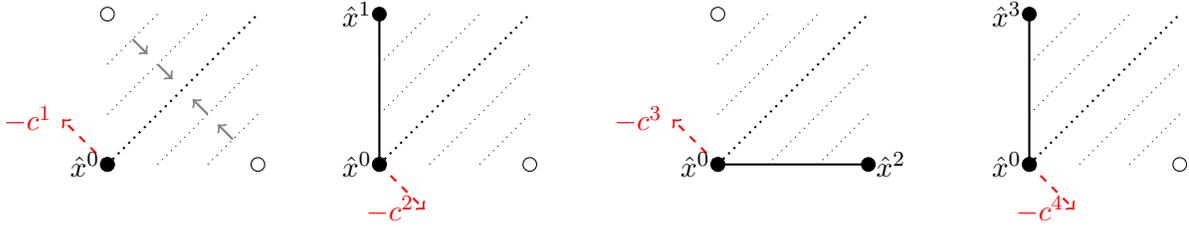

Considering this example, two questions may arise. Firstly, the
solution~$x^1$ is actually optimal, so that choosing a better
subgradient (namely zero) would have stopped the algorithm
immediately. Secondly, the scenarios $(0,1,-1)^\top$ and
$(0,-1,1)^\top$ defining the
uncertainty set~$U$ contain negative entries. The following example
shows that neither of the two features causes cycling:
\begin{example}\label{ex2}
 Let us consider the following problem  
\begin{equation*}
    \begin{array}{l l}
        \min & \max\{x_1, x_2\}\\[1ex]
        \textnormal{ s.t. } & x_1+ x_2 \geq1\\
        & x_1, x_2 \leq 1.
      \end{array}
\end{equation*}
Starting from $x^1={1\choose 1}$, Algorithm~{\texttt{SD-DROP}} will perform the following
iterations:
\begin{itemize}
 \item[k=1:] $x^1={1\choose 1}$, $V^1 = \{x^1\}$, $\alpha^1 = (1)$ and  $\partial f(x^1)\cap(-\mathcal{N}_{\conv(V^1)}(x^1))
 =\conv\{ {1 \choose 0}, {0 \choose 1}\}$.\\ We choose $c^1 = {1 \choose 0}$. Then $\hat x^1 ={0\choose 1}$, $V^2 = \{{1 \choose 1}, {0\choose 1}\}$
 \item[k=2:] $x^2$ may be ${1\choose 1}$ again, with $\alpha^2 = {1 \choose 0}$, so we eliminate ${0\choose 1}$, hence $V^2 = \{{1 \choose 1}\}=V^1$.
 \end{itemize}
 See Fig.~\ref{fig:ex2} for an illustration.
\end{example}
\begin{figure}
  \begin{center}
    \begin{tikzpicture}[scale=2]
      \draw[dotted,thick] (0,0) -- (1,1);
      \draw[dotted] (0,1) -- (1,1);
      \draw[dotted] (0,2/3) -- (2/3,2/3);
      \draw[dotted] (0,1/3) -- (1/3,1/3);
      \draw[dotted] (1,0) -- (1,1);
      \draw[dotted] (2/3,0) -- (2/3,2/3);
      \draw[dotted] (1/3,0) -- (1/3,1/3);
      \draw[->, thick, dashed, color=red] (1,1) -- (1-0.4,1) node[above] {$-c^1$};
      \draw[->, thick, gray] (0.15,1/3) -- (0.15,1/3-1/10);
      \draw[->, thick, gray] (0.15,2/3) -- (0.15,2/3-1/10);
      \draw[->, thick, gray] (0.15,1) -- (0.15,1-1/10);
      \draw[->, thick, gray] (1/3,0.15) -- (1/3-1/10,0.15);
      \draw[->, thick, gray] (2/3,0.15) -- (2/3-1/10,0.15);
      \draw[->, thick, gray] (1,0.15) -- (1-1/10,0.15);
      \node[scale=0.5, draw, circle, fill=black] at (1,1) {};
      \node[scale=0.5, draw, circle, fill=white] at (0,1) {};
      \node[scale=0.5, draw, circle, fill=white] at (1,0) {};
      \node[color=black] at (1+0.15,1) {$\hat x^0$};
    \end{tikzpicture}
    \qquad
    \begin{tikzpicture}[scale=2]
      \draw[dotted,thick] (0,0) -- (1,1);
      \draw[dotted] (0,1) -- (1,1);
      \draw[dotted] (0,2/3) -- (2/3,2/3);
      \draw[dotted] (0,1/3) -- (1/3,1/3);
      \draw[dotted] (1,0) -- (1,1);
      \draw[dotted] (2/3,0) -- (2/3,2/3);
      \draw[dotted] (1/3,0) -- (1/3,1/3);
      \draw[->, thick, gray] (0.15,1/3) -- (0.15,1/3-1/10);
      \draw[->, thick, gray] (0.15,2/3) -- (0.15,2/3-1/10);
      \draw[->, thick, gray] (0.15,1) -- (0.15,1-1/10);
      \draw[->, thick, gray] (1/3,0.15) -- (1/3-1/10,0.15);
      \draw[->, thick, gray] (2/3,0.15) -- (2/3-1/10,0.15);
      \draw[->, thick, gray] (1,0.15) -- (1-1/10,0.15);
      \node[scale=0.5, draw, circle, fill=black] at (1,1) {};
      \node[scale=0.5, draw, circle, fill=black] at (0,1) {};
      \node[scale=0.5, draw, circle, fill=white] at (1,0) {};
      \node[color=black] at (1+0.15,1) {$\hat x^0$};
      \node[color=black] at (-0.15,1) {$\hat x^1$};
      \draw[thick] (0,1) -- (1,1);
    \end{tikzpicture}
    \qquad
    \begin{tikzpicture}[scale=2]
      \draw[dotted,thick] (0,0) -- (1,1);
      \draw[dotted] (0,1) -- (1,1);
      \draw[dotted] (0,2/3) -- (2/3,2/3);
      \draw[dotted] (0,1/3) -- (1/3,1/3);
      \draw[dotted] (1,0) -- (1,1);
      \draw[dotted] (2/3,0) -- (2/3,2/3);
      \draw[dotted] (1/3,0) -- (1/3,1/3);
      \draw[->, thick, dashed, color=red] (1,1) -- (1-0.4,1) node[above] {$-c^2$};
      \draw[->, thick, gray] (0.15,1/3) -- (0.15,1/3-1/10);
      \draw[->, thick, gray] (0.15,2/3) -- (0.15,2/3-1/10);
      \draw[->, thick, gray] (0.15,1) -- (0.15,1-1/10);
      \draw[->, thick, gray] (1/3,0.15) -- (1/3-1/10,0.15);
      \draw[->, thick, gray] (2/3,0.15) -- (2/3-1/10,0.15);
      \draw[->, thick, gray] (1,0.15) -- (1-1/10,0.15);
      \node[scale=0.5, draw, circle, fill=black] at (1,1) {};
      \node[scale=0.5, draw, circle, fill=white] at (0,1) {};
      \node[scale=0.5, draw, circle, fill=white] at (1,0) {};
      \node[color=black] at (1+0.15,1) {$x^2$};
    \end{tikzpicture}
  \end{center}
  \caption{Illustration of Example~\ref{ex2}.}\label{fig:ex2}
\end{figure}
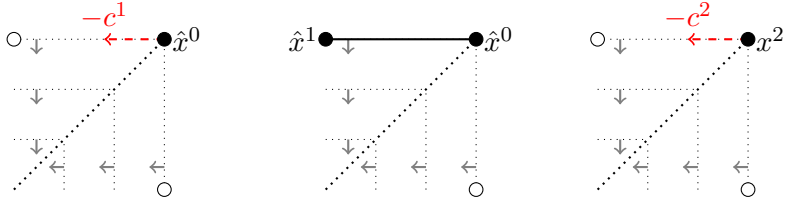

It is easy to see that cycling cannot occur when~$f$ is
differentiable. In fact, in this case, if~$x^k$ is not optimal for
Problem~\eqref{ContRel0}, we have $f(x^{k+1})< f(x^k)$, since~$-c^k$ is
a descent direction. In particular, Algorithm~{\texttt{SD-DROP}} terminates
after a finite number of iterations.

For the case of finite~$U$, differentiability is not given, and
cycling can occur as seen in the examples above. However, we can still
show a weaker result: we will prove that a small random perturbation
of the scenario entries can avoid cycling (with probability one). For this, we first need the following observation.
\begin{lemma}\label{lemma_descent}
  Let $L$ be an affine subspace of~$\R^n$ and assume that $f|_L$, the
  restriction of~$f$ to~$L$, is differentiable in~$x\in L$. Consider~$y\in
  L$ and~$c\in\partial f(x)$ with~$c^\top(y-x)<0$. Then~$y-x$ is a
  descent direction of~$f$ in~$x$.
\end{lemma}
\begin{proof}
  Let~$d:=y-x$. Then $\frac{\partial f}{\partial d}(x)=-\frac{\partial
    f}{\partial(-d)}(x)$, since~$f|_L$ is differentiable
  in~$x$. From~$c\in\partial f(x)$ we obtain~$\frac{\partial
    f}{\partial(-d)}(x)\ge c^\top (-d)$. Hence~$\frac{\partial
    f}{\partial d}(x)\le c^\top d<0$.
\end{proof}
\begin{lemma}\label{theorem_improve}
  Consider an iteration~$k$ in which Algorithm~{\texttt{SD-DROP}} does
  not terminate. Let~$L$ be an affine subspace of~$\R^n$
  containing~$x^k$ such that
  \begin{itemize}
  \item[(i)] $f|_L$ is differentiable in~$x^k$,\\[-3.75ex]
  \item[(ii)] $\dim (L\cap \aff(V^{k+1}))\ge 1$,\\[-3.75ex] 
  \item[(iii)] and $c^k$ is not orthogonal to~$L\cap \aff (V^{k+1})$.
  \end{itemize}
  Then~$f(x^{k+1})<f(x^k)$.
\end{lemma}
\begin{proof}
  By~(ii),
  there exists some $y\in
  L\cap \aff(V^{k+1})$ with~$y\neq x^k$. Since~$L$
  and~$\aff(V^{k+1})$ are affine spaces both containing~$x^k$, we may choose~$y$ such
  that~$(c^k)^\top (y-x^k)\le 0$.  By~(iii), we may even assume
  that~$(c^k)^\top (y-x^k)< 0$. Using Lemma~\ref{lemma_descent}
  and~(i), we thus derive that~$y-x^k$ is a descent direction of~$f$
  in~$x^k$. It thus remains to show that~$x^k+\varepsilon(y-x^k)\in
  \conv(V^{k+1})$ for some~$\varepsilon>0$, which implies that
  some~$\bar x\in\conv(V^{k+1})$ has a strictly smaller objective
  value than~$x^k$ and hence~$f(x^{k+1})\le f(\bar x)<f(x^k)$.

  Since~$y\in \aff(V^{k+1})$, we can write~$y=x^k+\sum_{v\in \bar
    V^k}\gamma_v(v-x^k)+\delta (\hat x^k-x^k)$, where we define~$\bar
  V^k:=\{v\in V^k\mid \alpha^k_v> 0\}$. As~$x^k$ belongs to the
  relative interior of~$\conv(\bar V^{k})$, there
  exists~$\bar\varepsilon>0$ such that $x^k+\sum_{v\in \bar
    V^k}\bar\varepsilon\gamma_v(v-x^k)\in \conv(\bar V^k)$.
  From~$(c^k)^\top (y-x^k)< 0$, $(c^k)^\top \hat x^k<(c^k)^\top x^k$,
  and~$(c^k)^\top (v-x^k)= 0$ for all~$v\in \bar V^k$ we
  derive~$\delta> 0$.  By choosing~$\bar\varepsilon\le\tfrac 1\delta$,
  we may assume that~$x^k+\bar\varepsilon\delta(\hat x^k-x^k)\in
  \conv(V^{k+1})$. Altogether, we derive that~$x^k+\tfrac
  12\bar\varepsilon(y-x^k)\in\conv(V^{k+1})$.
\end{proof}
\begin{theorem}\label{theorem_perturb}
  Assume that all scenarios in~$U=\{c_1,\dots,c_m\}$ are perturbed by
  any continuously distributed random vector in~$\R^{m(n+1)}$ with
  full-dimensional support. Then, with probability one,
  Algorithm~{\texttt{SD-DROP}} terminates after finitely many iterations.
\end{theorem}
\begin{proof}
  Consider any iteration~$k$ in which Algorithm~{\texttt{SD-DROP}} does not
  terminate. Then it suffices to show that~$f(x^{k+1})<f(x^k)$ with
  probability one. For this, let~$L$ be the maximal affine space such
  that~$x^k\in L$ and~$f|_L$ is differentiable in~$x^k$.  By the
  definition of~$f$, the epigraph~$\epi(f)$ is a polyhedron, and~$L$
  is obtained by projecting the minimal face of~$\epi(f)$
  containing~$(x^k,f(x^k))$ onto~$\R^n$ and taking the affine hull of
  the projection.  Thus~$L$ is an affine subspace of~$\R^n$
  containing~$x^k$, depending continuously
  on the perturbation of~$c_1,\dots,c_m$.
  We claim that the
  conditions~(ii) and~(iii) of Lemma~\ref{theorem_improve} are
  satisfied with probability one then, so that the result follows.
    Indeed, using the same notation as in the proof of Theorem~\ref{th:c_unique}, we have
    $\dim (L) = n- \dim (\partial f(x^k))= n- ( |C^=|-1)= n - (|V^k|-|A^=|-1)=n-\dim(\aff \bar{V}^k)$ with probability one, since $|C^=|+|A^=|=|V^k|$ with probability one. Thus, with probability one, we obtain $\dim(L) + \dim(\aff \bar{V}^k)=n$ and hence $\dim(L) + \dim(\aff V^{k+1})\geq n+1$. 
   Thus~(ii) holds with probability one.
  For showing~(iii), we use again that $\dim (L\cap
  \aff V^{k+1})\ge 1$ with probability one. This implies that the probability
  of the fixed vector~$c^k$ being orthogonal to~$L\cap \aff V^{k+1}$ is zero.
\end{proof}
Theorem~\ref{theorem_perturb} shows that cycling can be avoided by
applying small random perturbations to the scenarios~$c_1,\dots,c_m$,
e.g., by choosing~$(\hat c_j)_i\in
[(c_j)_i-\varepsilon,(c_j)_i+\varepsilon]$ uniformly at random for
some~$\varepsilon>0$, independently for all~$j=1,\dots,m$
and~$i=0,\dots,n$. As~$X$ is finite, the optimal solution of the
perturbed problem~\eqref{MinMaxProb} will agree with an optimizer of
the unperturbed problem if~$\varepsilon$ is small enough (even though
this is not true for the relaxation~\eqref{ContRel0}). Note that
Theorem~\ref{theorem_perturb} requires that also the constant in the
objective function is perturbed.
\begin{remark}
  In practice, the perturbation applied in
  Theorem~\ref{theorem_perturb} is not necessary, because small
  numerical errors arising in the optimization process will have
  the same effect. In our experiments described in
  Section~\ref{section:results}, we do not explicitly apply any
  perturbation.
\end{remark}
 
Note that Theorem~\ref{th:c_unique} also holds when eliminating
vertices. In particular, this implies that, when starting from the
same set~$V^k$, the next subgradient~$c^k$ is the same with or without
elimination (with high probability). However, in a later iteration,
the set~$V^k$ and hence the optimal solution~$x^k$ may be different in
the two cases, and thus also the subgradients.
In our experiments, we observed that vertices being eliminated were sometimes re-generated in subsequent iterations.

\section{Embedding \texttt{SD} into a branch-and-bound scheme}\label{sec:Embed_SD}

In order to solve Problem~\eqref{MinMaxProb} to optimality, we embed
Algorithm~\texttt{SD} into a branch-and-bound scheme, which we will
denote by~\texttt{BB-SD}. In this branch-and-bound scheme, we can
exploit some specific features of the SD algorithm. Firstly, all
binary vertices generated at some node of the branch-and-bound tree
can be reused in the child nodes. Indeed, if we branch on fractional
variables, each such vertex must be feasible in one of the child
nodes, and can thus be inherited. This initial set of vertices enables
us to \emph{warmstart} the SD algorithm at every child node.
Moreover, thanks to Lemma~\ref{lemma_lb}, every iteration of
\texttt{SD} leads to a valid lower bound on the solution of the convex
relaxation considered, meaning that \emph{early pruning} can be
performed. In other words, at every node we do not necessarily need to
solve the convex relaxation to optimality: the node can be pruned as
soon as \texttt{SD} computes a lower bound greater than the current
upper bound.

As emphasized above, we assume that Problem~\eqref{CertainProb} can be
accessed only by an optimization oracle.  Therefore, even when dealing
with specific combinatorial problems, we do not exploit any structure
to define primal heuristics within \texttt{BB-SD}.  Nevertheless, at
every node of \texttt{BB-SD}, we easily get an upper bound by
evaluating the objective function on all the generated extreme points
and taking the minimal value among them.

Having no problem-specific
heuristics, we need an enumeration strategy that provides primal
solutions quickly. Thus, we adopt a \emph{depth first search
  (DFS)}. Moreover, the branching rule implemented within
\texttt{BB-SD} branches on variables that are fractional in the
continuous relaxation, by means of the canonical disjunction. More
precisely, we branch on the variable with the fractional part closest
to one and produce two child nodes: in the node considered first, we
fix the branching variable to $1$, in the other node we fix it to $0$.
This choice, combined with DFS, typically allows to quickly find
integer solutions, which are sparse for many combinatorial problems.
Note that all nodes in \texttt{BB-SD} remain feasible. Indeed,
regardless of how we select the fractional variable~$x_i$ to branch
on, the oracle must have already returned solutions with both~$x_i=0$
and~$x_i=1$.

\section{Numerical results}\label{section:results}

To test the performance of our algorithm~\texttt{SD} and of the
branch-and-bound scheme~\texttt{BB-SD}, we considered instances of
Problem~\eqref{MinMaxProb} with two different underlying problems: the
Spanning Tree problem (Section~\ref{sec:num-stp}) and the Traveling
Salesman problem (Section~\ref{sec:num-tsp}).  The standard models for
these problems use an exponential number of constraints that can be
separated efficiently. In the case of the Spanning Tree problem, this
exponential set of constraints yields a complete linear formulation,
while this is not the case for the NP-hard Traveling Salesman problem.
For the robust Minimum Spanning Tree Problem, we report a comparison
between~\texttt{BB-SD} and the MILP solver of \texttt{CPLEX}~\cite{cplex-url}. 
For the robust Traveling Salesman Problem, we focus on the continuous
relaxations, thus reporting a comparison on the bounds obtained at the
root node of the branch-and-bound tree.
 
In the implementation of \texttt{SD}, for both the robust Minimum
Spanning Tree Problem (r-MSTP) and the robust Traveling Salesman
Problem (r-TSP), oracle LIN-O is defined according to the underlying
problem: for the r-MSTP we implemented the standard Kruskal
algorithm~\cite{kruskal1956shortest}, a well-known polynomial-time
algorithm. For the r-TSP, we used the implementation of the solver
\texttt{Concorde}~\cite{concorde}.  Since the TSP is NP-hard,
the computational times needed to call the linear oracles differ
significantly in the two problems, as seen later in the numerical
experiments.

Except for the oracle LIN-O, we used exactly the same implementation
for both problems. In particular, we applied the same oracle SIM-O for
both r-MSTP and r-TSP. Problem~\eqref{eq:probdiscrete} is rewritten by
expanding the condition $x\in \conv(V^k)$. By using the LP
formulation~\eqref{eq:probdiscrete2} and eliminating the $x$
variables, we obtain the following equivalent formulation:
 \begin{equation}\label{eq:probdiscrete_alpha}
\begin{array}{l l}
        \min & z\\[1ex]
        \textnormal{ s.t. }& \sum_{v\in V^k} \tilde{c}_j^v \alpha_v \leq z,\quad j=1,\ldots,m\\[1.2ex]
        & \sum_{v\in V^k} \alpha_v=1 \\[1.2ex]
        & \alpha_v\geq 0,\quad  v\in V^k,
\end{array}
\end{equation}
where $\tilde{c}_j^v=c_j^\top v$, for every $v\in V^k$.
Problem~\eqref{eq:probdiscrete_alpha} is solved with \texttt{CPLEX}.  Note that
the number of constraints depends on the number of scenarios, and the
number of variables corresponds to the cardinality of $V^k$ and thus
increases at every iteration in our approach~\texttt{SD}. The
dropping rule implemented in~\texttt{SD-DROP} may reduce the size of this
problem, thus potentially leading to practical improvements in the
running time.

\subsection{Spanning Tree Problem}\label{sec:num-stp}

Given an undirected weighted graph~$G=(N,E)$, a minimum spanning tree
is a subset of edges that connects all vertices, without any cycles
and with the minimum total edge weight. We use the following
formulation of the Robust Minimum Spanning Tree problem:
\begin{equation}\label{mst}\tag{r-MSTP}
  \begin{array}{rrcll}
    \min & \multicolumn{4}{l}{\max_{c\in U} c^\top x}\\[1.5ex]
    \textnormal{ s.t.} & \sum_{(u,v)\in E}x_{u,v} & = & |N|-1  \\[1ex]
    & \sum_{u,v\in X} x_{u,v} & \le & |X|-1 & \forall \,\emptyset\neq X\subseteq N\\[1ex]
    & x & \in & \{0,1\}^{E}
  \end{array}
\end{equation}
The objective function can easily be linearized by introducing a new
variable~$z\in\R$ and constraints~$z\ge c^\top x$ for all~$c\in U$.
In the above model, the number of inequalities is exponential in the 
input size, hence we have to use a separation algorithm
within~\texttt{CPLEX}.

For our experiments, we consider a benchmark of randomly generated
instances of r-MSTP.  We build complete graphs of five different sizes
(from $20$ to $60$ nodes).  The nominal costs are real numbers
randomly chosen in the interval $[1,2]$. For each size we randomly
generate $10$ different nominal cost vectors.  The scenarios $c\in U$
are generated by adding to the vector of nominal costs a random unit
vector, multiplied by a scalar factor $\beta$. We consider three
such factors~$1$,~$2$, and~$3$, and generate three different numbers of
scenarios (\#sc) $10$, $100$, and~$1000$. In total, then, we have a benchmark of
450 instances.

In a first experiment, we analyze different dropping rules within
algorithm \texttt{SD}.  Then, we show how performing a warm start
along the branch-and-bound iterations leads to a significant reduction
in terms of number of iterations compared to a cold start.  Finally,
our branch-and-bound method \texttt{BB-SD} is compared to the
MILP solver of \texttt{CPLEX}. Within \texttt{CPLEX}, we apply a
dynamic separation algorithm using a callback adding lazy constraints,
adopting a simple implementation based on the Ford-Fulkerson
algorithm.

\subsubsection{Comparison of dropping rules}

In this section, we focus on the performance of algorithm \texttt{SD}
for solving the continuous relaxation of our r-MSTP instances.  In
particular, we compare the performance of \texttt{SD} implementing
three different dropping rules.
Since all instances with 10 or 100 scenarios are solved in less than $0.1$~seconds, we only consider the
continuous relaxations of instances with $1000$ scenarios,
meaning that the evaluation is carried out on $150$ instances.
As dropping rules, we implemented the
following:
\begin{itemize}
\item \texttt{d0}: meaning that no dropping rule is applied within
  \texttt{SD}
\item \texttt{d1}: meaning that we update the set $V^k$ according to
  condition~\eqref{eq:elrule} defined in
  Section~\ref{sec:eliminate}, i.e., we eliminate all vertices $v\in
  V^k$ such that $\alpha_v^k=0$ at every iteration of \texttt{SD}
\item \texttt{d2}: meaning that we eliminate vertices $v\in V^k$ such that
  $\alpha_v^k=0$ \emph{only if} they provide a strict ascent
  direction, i.e., if~$(c^k)^\top (v-x^k)\geq
  \varepsilon$ with a fixed threshold $\varepsilon>0$.
\end{itemize}
As mentioned before, for each~$|N|$ 
we built $30$
instances, $10$ for each factor~$\beta$. In Table~\ref{tab:elimCR}, we report
the average running times (time) and the average numbers of iterations
(\#it) for each version of~\texttt{SD}; note that the number~$n$ of variables in~\eqref{CertainProb} is given by~$|E|$ here.
\begin{figure}[H]
    \centering
    \includegraphics[trim={1.5cm 0.5cm 1cm 0},clip,width=0.45\textwidth]{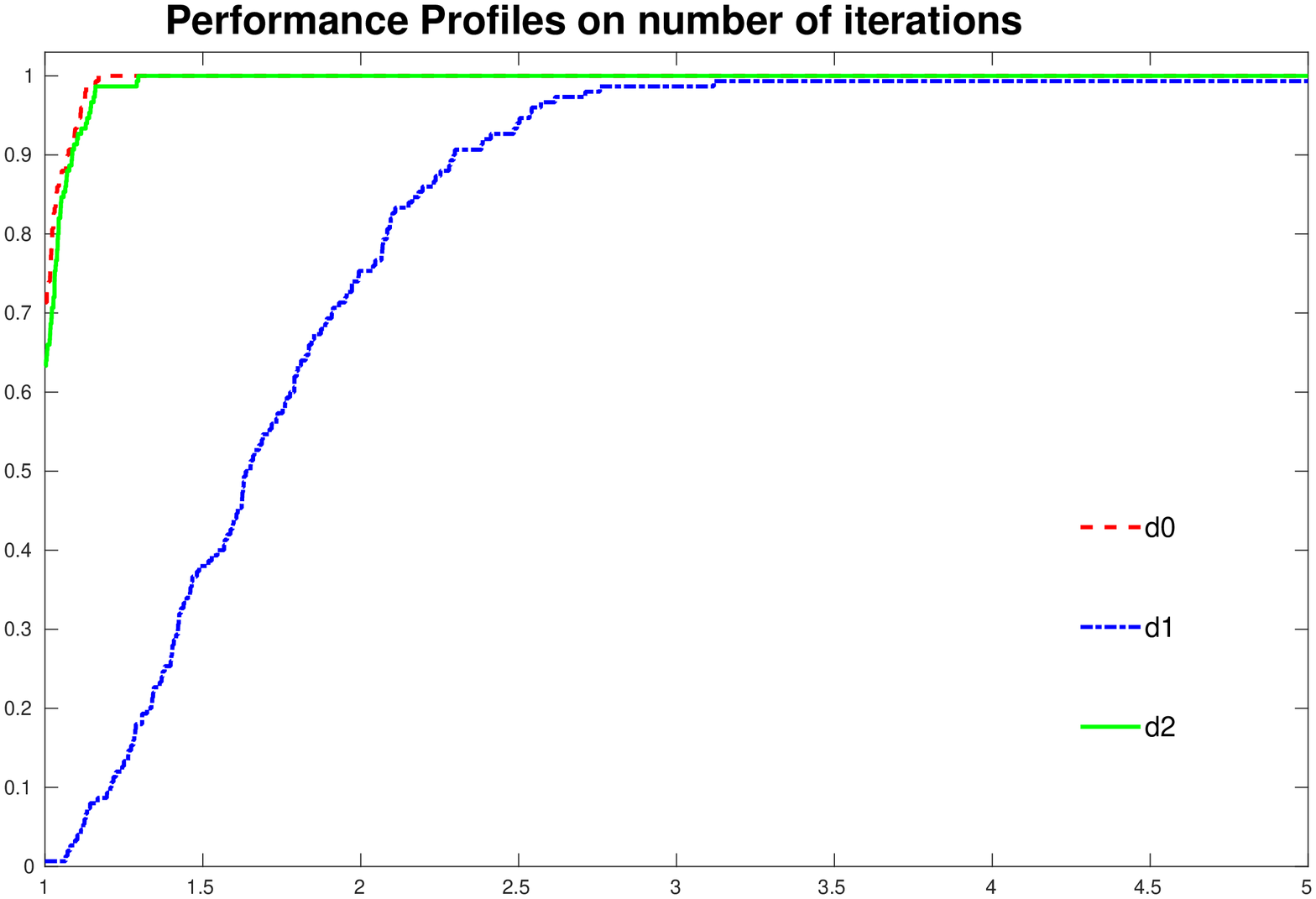}
    \includegraphics[trim={1.5cm 0.5cm 1cm 0},clip,width=0.45\textwidth]{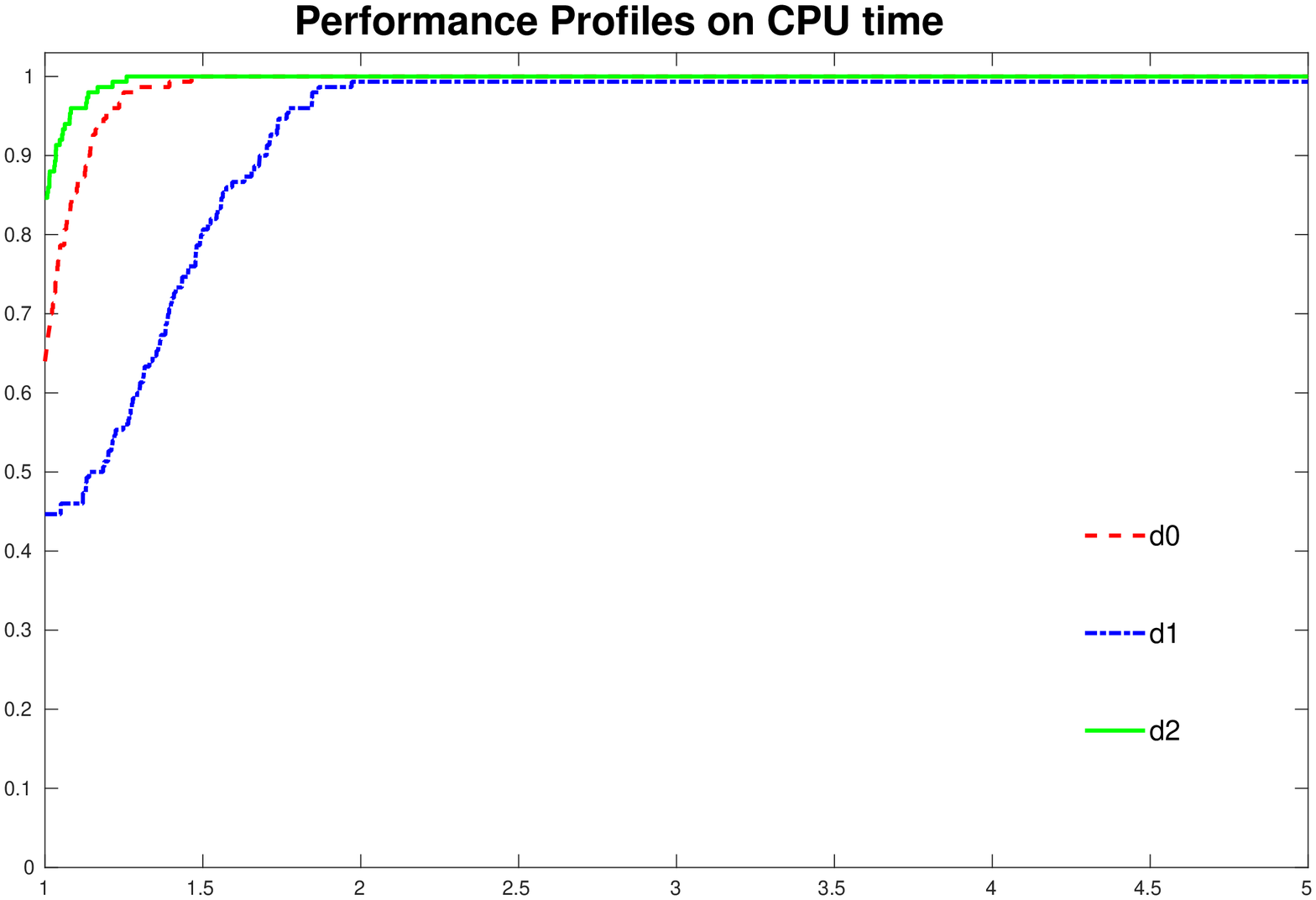}
    \caption{Comparison between different dropping rules ($1000$ scenarios).}
    \label{fig:elimCR}
\end{figure}
\begin{table}
  \centering      
  {\scalebox{.8}{
  \begin{tabular}{|cc||r|r||r|r||r|r|} 
    \hline
    \multicolumn{2}{|c||}{} & \multicolumn{2}{c||}{\texttt{d0}}  & \multicolumn{2}{c||}{\texttt{d1}} 
    & \multicolumn{2}{c|}{\texttt{d2}} \\
    \hline
    \hline
    $|N|$	&	$|E|$	& time	&	\#it	&	time	&	\#it	&	time	&	\#it	\\
    \hline
20	&	190	&0.27	&	69.3	&	0.30	&	105.0	&	0.25	&	73.1	\\
30	&	435	&0.85	&	111.1	&	0.95	&	184.4	&	0.73	&	110.6	\\
40	&	780	&1.72	&	152.2	&	2.31	&	295.2	&	1.59	&	155.6	\\
50	&	1225	&2.55	&	173.6	&	3.74	&	353.6	&	2.44	&	172.4	\\
60	&	1770	&3.36	&	208.5	&	4.65	&	416.6	&	3.28	&	205.5	\\
\hline
  \end{tabular}\vspace*{0.1cm}
  }}
  \caption{Comparison between different dropping rules ($1000$ scenarios).}
  \label{tab:elimCR}  
\end{table}
We further report the performance profiles~\cite{DM2002} related to
the CPU time and the number of iterations in Fig.~\ref{fig:elimCR}.
We notice that dropping rule \texttt{d2} allows \texttt{SD} to have
slightly better performance in terms of CPU time, despite being 
not always better with respect to \texttt{SD} with no dropping rule in
terms of number of iterations.
  Indeed, in the MST problem, the linear oracle calls are
extremely fast, while most of the computational time is needed to
solve the master problem~\eqref{eq:probdiscrete_alpha}. 
This explains why, although some
more iterations are needed, dropping some vertices and hence reducing
the size of the master problems can improve the overall performance of
the algorithm.  On the other hand, it is clear from the results that
eliminating all inactive vertices as in Algorithm SD-DROP (rule
\texttt{d1}) is not beneficial for \texttt{SD} as both the number of
iterations and the CPU time increase.

\subsubsection{Warmstart benefits}

In the following, we evaluate our branch-and-bound method
\texttt{BB-SD}. In particular, we will compare the performance
of~\texttt{BB-SD} considering both \texttt{SD} with no dropping rule
(\texttt{d0}) and \texttt{SD} with dropping rule (\texttt{d2}).  The
comparison is done on all $450$ instances of r-MST. As mentioned
before, for each combination of $|N|$ and \#sc, 
we built $30$ instances, $10$ for each value of the scalar factor $\beta$. 
In Table~\ref{tab:WSe0}, we first compare the performance of
\texttt{BB-SD} by considering \texttt{SD} with no dropping rule with
and without warmstart (\texttt{d0 no ws} vs \texttt{d0 with ws}).  In
the table, we report the number of instances solved within the time
limit of one hour (\#sol),
the average running times (time),
the average number of nodes (\#nodes)
and the average number of iterations (\#it) for each version
of~\texttt{SD}. All averages are taken over 
the instances solved within the time limit. 
It is clear from the results that the warmstart leads
to a considerable decrease in the average number of iterations and
consequently a significant decrease in CPU time.
The same behavior can be noticed when looking at Table~\ref{tab:WSe2},
where we compare \texttt{BB-SD} using
dropping rule \texttt{d2} with and without warmstart (\texttt{d2 no
  ws} vs \texttt{d2 with ws}).
\begin{table}
  \centering      
  {\scalebox{.8}{
  \begin{tabular}{|ccc||r|r|r|r||r|r|r|r|} 
    \hline
    \multicolumn{3}{|c||}{} & \multicolumn{4}{c||}{\texttt{d0 no ws}}  & \multicolumn{4}{c|}{\texttt{d0 with ws}} 
    \\
    \hline
    \hline
     $|N|$	&	$|E|$ & \#sc &  \#sol& time	& \#nodes &	\#it & \#sol &	time	& \#nodes &	\#it	\\
    \hline
     20 &  190 & 10   & 30 &     1.49 &      6.72e+2 &      8.00e+3 &  30 &     0.84 &      6.69e+2 &      4.23e+3 \\ 
        &      & 100  & 30 &    60.03 &      6.85e+3 &      1.63e+5 &  30 &    33.44 &      6.87e+3 &      8.56e+4 \\
        &      & 1000 & 24 &   476.59 &      6.74e+3 &      1.66e+5 &  27 &   588.38 &      1.24e+4 &      1.78e+5 \\
    \hline 
    30 &     435 & 10   &   30 &    13.12 &     3.15e+3 &      5.97e+4 &  30 &     7.34 &      3.52e+3 &      3.21e+4 \\
       &         & 100  &  25 &   602.07 &      3.55e+4 &      1.15e+5 &  27 &   451.05 &      5.04e+4 &      8.29e+5 \\
       &         & 1000 &  11 &   493.33 &      3.47e+3 &      8.30e+4 &  14 &   642.41 &      8.64e+3 &      1.12e+5 \\
    \hline 
    40 &    780 & 10 & 30 &   110.32 &      1.59e+4 &      4.10e+5 &  30 &    51.77 &      1.59e+4 &      1.91e+5 \\
       &        & 10 & 17 &   764.88 &      3.24e+4 &      1.10e+6 &  19 &   632.17 &      5.25e+4 &      8.65e+5 \\
       &        & 10 &  8 &  1592.79 &      7.67e+3 &      1.78e+5 &   9 &   837.77 &      9.12e+3 &      9.24e+4 \\
    \hline 
    50 &    1225 & 10  & 30 &   441.38 &      4.35e+4 &      1.32e+6 &  30 &   201.05 &      4.40e+4 &      6.10e+5 \\
       &         & 100 & 10 &   137.71 &      7.26e+3 &      1.53e+5 &  13 &   558.02 &      3.47e+4 &      5.81e+5 \\
       &         & 1000&  7 &  1209.52 &      3.83e+3 &      9.81e+4 &   9 &   810.38 &      5.92e+3 &      6.64e+4 \\
    \hline
    60 &    1770 & 10   & 25 &   513.71 &      4.21e+4 &      1.28e+6 &  30 &   575.38 &      9.81e+4 &      1.47e+6 \\
       &         & 100  &   9 &   240.12 &      9.04e+3 &     2.09e+5 &  12 &   689.81 &      3.50e+4 &      5.77e+5 \\
       &         & 1000 &  2 &   867.80 &      2.08e+3 &      5.77e+4 &   3 &  1136.77 &      6.47e+3 &      7.58e+4 \\
    \hline
  \end{tabular}\vspace*{0.1cm}
  }}
  \caption{Using~\texttt{BB-SD} with and without warmstart.}
  \label{tab:WSe0}  
\end{table}
\begin{table}
  \centering  
  {\scalebox{.8}{
  \begin{tabular}{|ccc||r|r|r|r||r|r|r|r|} 
    \hline
    \multicolumn{3}{|c||}{} & \multicolumn{4}{c||}{\texttt{d2 no ws}}  & \multicolumn{4}{c|}{\texttt{d2 with ws}} 
    \\
    \hline
    \hline
    $|N|$	&	$|E|$	& \#sc & \#sol& time	& \#nodes &	\#it & \#sol &	time	& \#nodes &	\#it	\\
    \hline
    20 &     190 &  10   & 30 &     1.55 &      6.75e+2 &      8.25e+3 &  30 &     0.88 &      6.82e+2 &      4.39e+3 \\
       &         &  100  & 30 &    61.11 &      6.85e+3 &      1.69e+5 &  30 &    34.38 &      6.88e+3 &      8.82e+4 \\
       &         &  1000 & 24 &   480.61 &      6.74e+3 &      1.70e+5 &  27 &   592.45 &      1.24e+4 &      1.82e+5 \\ 
   \hline
   30 &     435 & 10   & 30 &    14.62 &      3.39e+3 &      6.52e+4 &  30 &     7.58 &      3.52e+3 &      3.24e+4 \\ 
      &         & 100  & 25 &   614.89 &      3.55e+4 &      1.16e+6 &  27 &   461.70 &      5.04e+4 &      8.34e+5 \\
      &         & 1000 & 12 &   695.62 &      4.76e+3 &      1.27e+5 &  14 &   576.20 &      7.51e+3 &      9.94e+4 \\
   \hline
   40 &     780 & 10   & 30 &   115.76 &      1.57e+4 &      4.10e+5 &  30 &    53.22 &      1.56e+4 &      1.88e+5 \\
      &         & 100  & 17 &   776.65 &      3.21e+4 &      1.08e+6 &  18 &   473.66 &      3.88e+4 &      6.37e+5 \\
      &         & 1000 & 7 &  1266.34 &       6.69e+3 &      1.42e+5 &   9 &   856.06 &      9.13e+3 &      9.26e+4 \\
   \hline
   50 &    1225 & 10   & 30 &   523.91 &      4.73e+4 &      1.44e+6 &  30 &   229.78 &      4.71e+4 &      6.52e+5 \\
      &         & 100  & 10 &   143.25 &      7.26e+3 &      1.53e+5 &  13 &   580.32 &      3.41e+4 &      5.75e+5 \\
      &         & 1000 & 7 &  1194.76 &      3.83e+3 &       9.81e+4 &   9 &   848.74 &      5.92e+3 &      6.64e+4 \\
   \hline
   60 &    1770 & 10   & 25 &   592.07 &      4.26e+4 &      1.29e+6 &  30 &   584.95 &      9.07e+4 &      1.36e+6 \\
      &         & 100  & 9 &   249.59 &      9.04e+3 &      2.09e+5 &  12 &   718.41 &      3.50e+4 &      5.77e+5 \\
      &         & 1000 & 2 &   859.83 &      2.08e+3 &      5.77e+4 &   3 &  1188.03 &      6.47e+3 &      7.58e+4 \\
   \hline
  \end{tabular}\vspace*{0.1cm}
  }}
  \caption{Using~\texttt{BB-SD} with and without warmstart, applying dropping rule \texttt{d2}.}
  \label{tab:WSe2}  
\end{table}

In Fig.~\ref{fig:compWStime}, we further report the performance
profiles with respect to the CPU time of the four versions
of~\texttt{BB-SD}. The versions of~\texttt{BB-SD} with no elimination
(\texttt{d0}) and with dropping rule \texttt{d2} show very similar
performances. Note that from our results it is clear that the
instances become harder with a higher number of scenarios. We can
also notice that~\texttt{BB-SD} with \texttt{d2 with ws} shows
slightly better performances when looking at the hardest instances,
namely those with~$1000$ scenarios.
\begin{figure}
    \centering
    \includegraphics[trim={1.5cm 0.5cm 1cm 0},clip,width=0.65\textwidth]{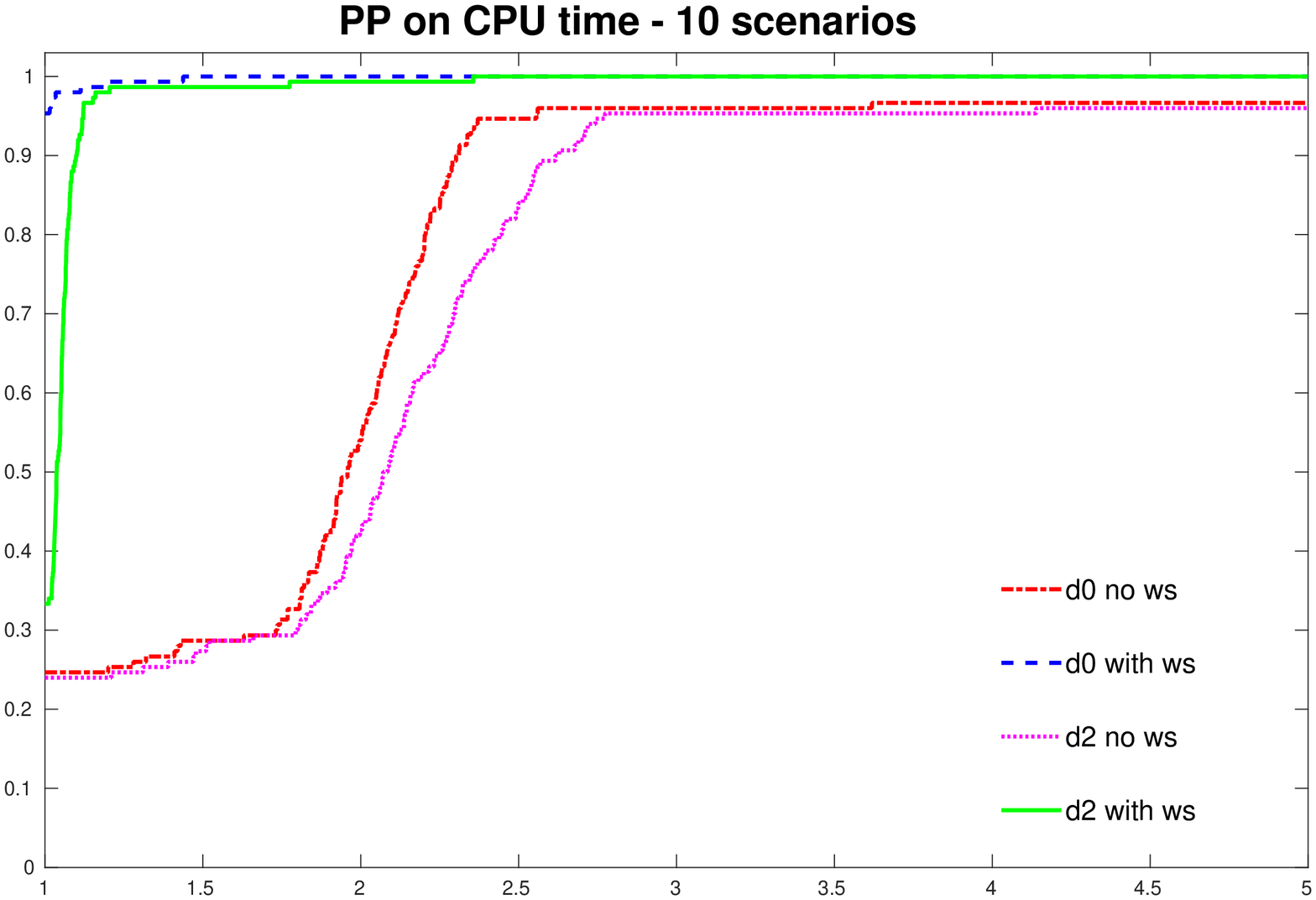}
    \includegraphics[trim={1.5cm 0.5cm 1cm 0},clip,width=0.65\textwidth]{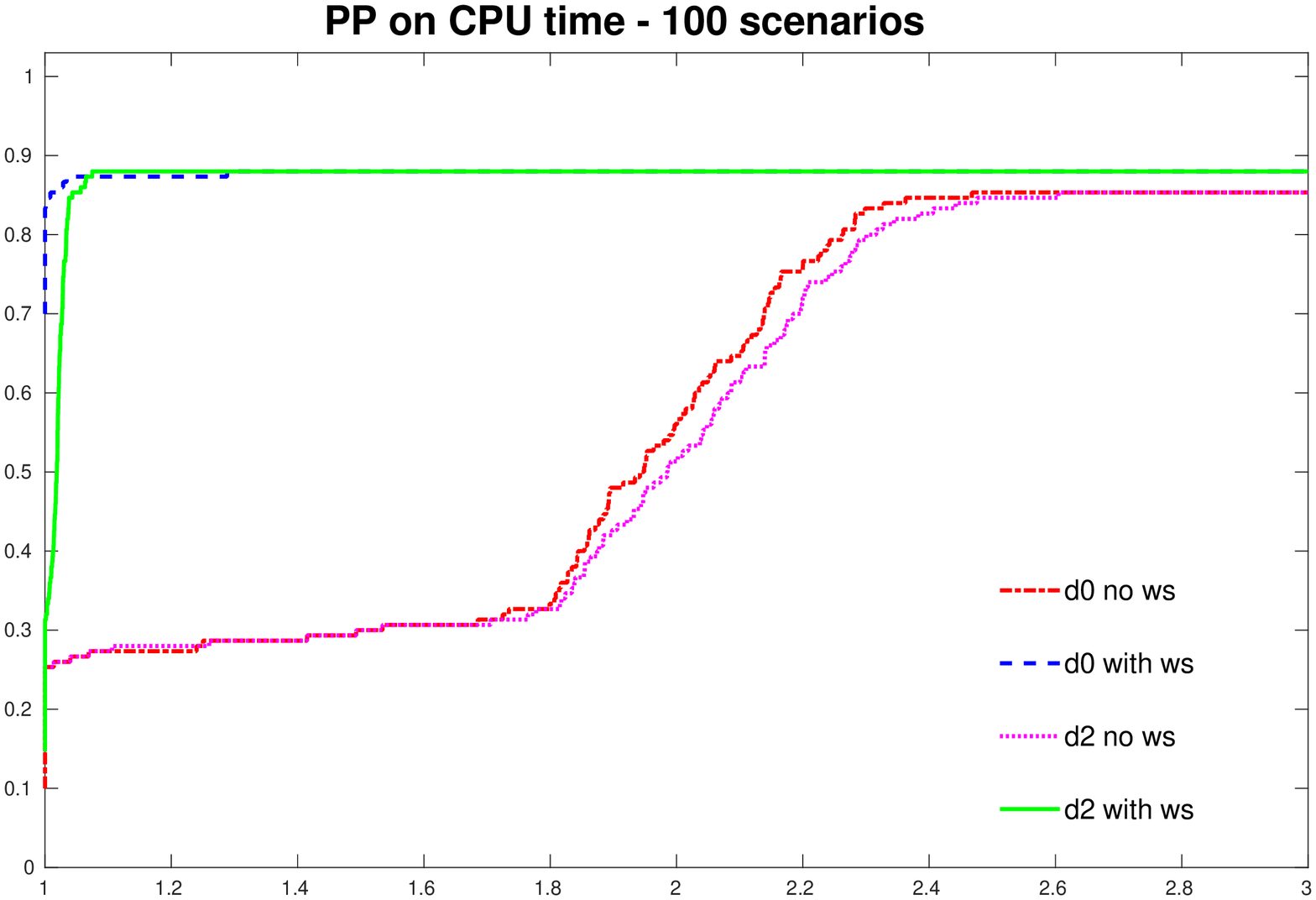}
    \includegraphics[trim={1.5cm 0.5cm 1cm 0},clip,width=0.65\textwidth]{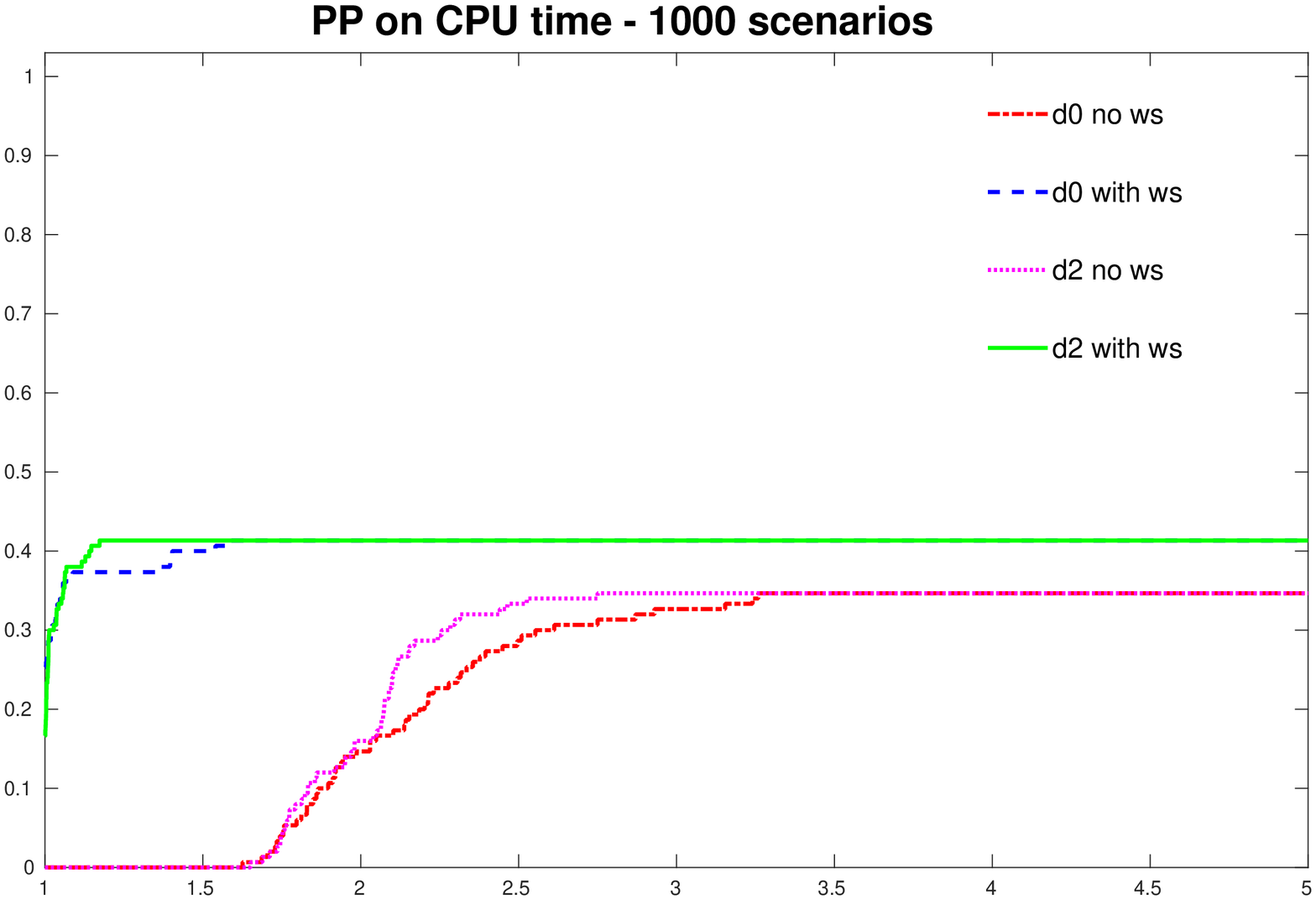}
    \caption{Comparison of \texttt{BB-SD} with and without warmstart.}
    \label{fig:compWStime}
\end{figure}

\subsubsection{Comparison between BB-SD and CPLEX}

We now compare our branch-and-bound algorithm~\texttt{BB-SD} with
\texttt{CPLEX} on our r-MSTP instances. As already mentioned, within
the MILP solver of \texttt{CPLEX}, we apply a dynamic separation
algorithm using a callback adding lazy constraints, adopting a simple
implementation based on the Ford-Fulkerson algorithm.  The comparison
is made with \texttt{BB-SD} implementing the dropping rule~\texttt{d2}
and allowing warmstart.  In Table~\ref{tab:WS}, we report for each
combination of $|N|$ and \#sc, 
the number of
instances solved within the time limit of one hour (\#sol), the average running
times (time) and the average number of nodes (\#nodes). We recall that
the averages are taken considering the results on $30$ instances each.
Performance profiles are presented in Fig.~\ref{fig:compCPLEX}.

We can notice that \texttt{BB-SD} strongly ouperforms \texttt{CPLEX}
when considering instances with~$10$ and~$100$ scenarios.  As already
highlighted before, the higher the number of scenarios, the harder the
instances become. Still, also when dealing with instances
containing~$1000$ scenarios, \texttt{BB-SD} shows better performance
with respect to \texttt{CPLEX}. On instances with~$10$ scenarios, we
see that~\texttt{BB-SD} is able to solve all instances within the time
limit, while~\texttt{CPLEX} fails in~$46$ cases. For instances on~$100$
scenarios,~\texttt{BB-SD} and~\texttt{CPLEX} show~$50$ and
$75$~failures, respectively, while for instances with~$1000$
scenarios,~\texttt{BB-SD} and~\texttt{CPLEX} show~$88$ and~$95$
failures.
 \begin{table}
	\centering   
	{\scalebox{.8}{
	\begin{tabular}{|ccc||r|r|r||r|r|r|} 
	  \hline
	  \multicolumn{3}{|c||}{}
	  & \multicolumn{3}{c||}{\texttt{BB-SD}}  
	  & \multicolumn{3}{c|}{\texttt{CPLEX}}\\
 	  \hline
	  \hline
	  $|N|$	&	$|E|$	& \#sc &  \#sol& time	&	\#nodes& \#sol	&	time	&	\#nodes\\
	  \hline
	  20 &     190 & 10   & 30 &     0.88 &      6.81e+2 &  30 &     0.36 &      2.15e+3 \\
	     &         & 100  & 30 &    34.38 &      6.88e+3 &  30 &    11.85 &      3.65e+4 \\
	     &         & 1000 & 27 &   592.45 &      1.24e+4 &  30 &   295.54 &      1.47e+5 \\
	  \hline
	  
	  30 &     435 & 10   & 30 &     7.58 &      3.52e+3 &  30 &    39.19 &      6.80e+4 \\
	     &         & 100  & 27 &   461.70 &      5.04e+4 &  27 &   399.98 &      7.91e+5 \\
	     &         & 1000 & 14 &   576.20 &      7.51e+3 &  17 &   649.76 &      2.42e+5 \\
	  \hline
	  
	  40 &     780 & 10   & 30 &    53.22 &      1.56e+4 &  21 &   304.25 &      3.56e+5 \\
	     &         & 100  & 18 &   473.66 &      3.88e+4 &  11 &   897.02 &      9.29e+5 \\
	     &         & 1000 &  9 &   856.06 &      9.13e+3 &   4 &  1159.91 &      3.29e+5 \\
	  \hline
	  
	  50 &    1225 & 10   & 30 &   229.78 &      4.71e+4 &  12 &   855.83 &      7.43e+5 \\
	     &         & 100  & 13 &   580.32 &      3.41e+4 &   4 &   685.22 &      5.17e+5 \\
	     &         & 1000 &  9 &   848.74 &      5.92e+3 &   3 &  3247.92 &      5.36e+5 \\
	  \hline
	  
	  60 &    1770 & 10   & 30 &   584.95 &      9.07e+4 &  11 &   596.26 &      4.24e+5 \\
	     &         & 100  & 12 &   718.41 &      3.50e+4 &   3 &  1167.98 &      9.92e+5 \\
	     &         & 1000 &  3 &  1188.03 &      6.47e+3 &   1 &   577.20 &      8.76e+4 \\
	  
	  \hline
	\end{tabular}\vspace*{0.1cm}
	}}
	\caption{Comparison between \texttt{BB-SD} and \texttt{CPLEX} on r-MST instances.}
	\label{tab:WS}  
 \end{table}

\begin{figure}
    \centering
    \includegraphics[trim={1.5cm 0.5cm 1cm 0},clip,width=0.49\textwidth]{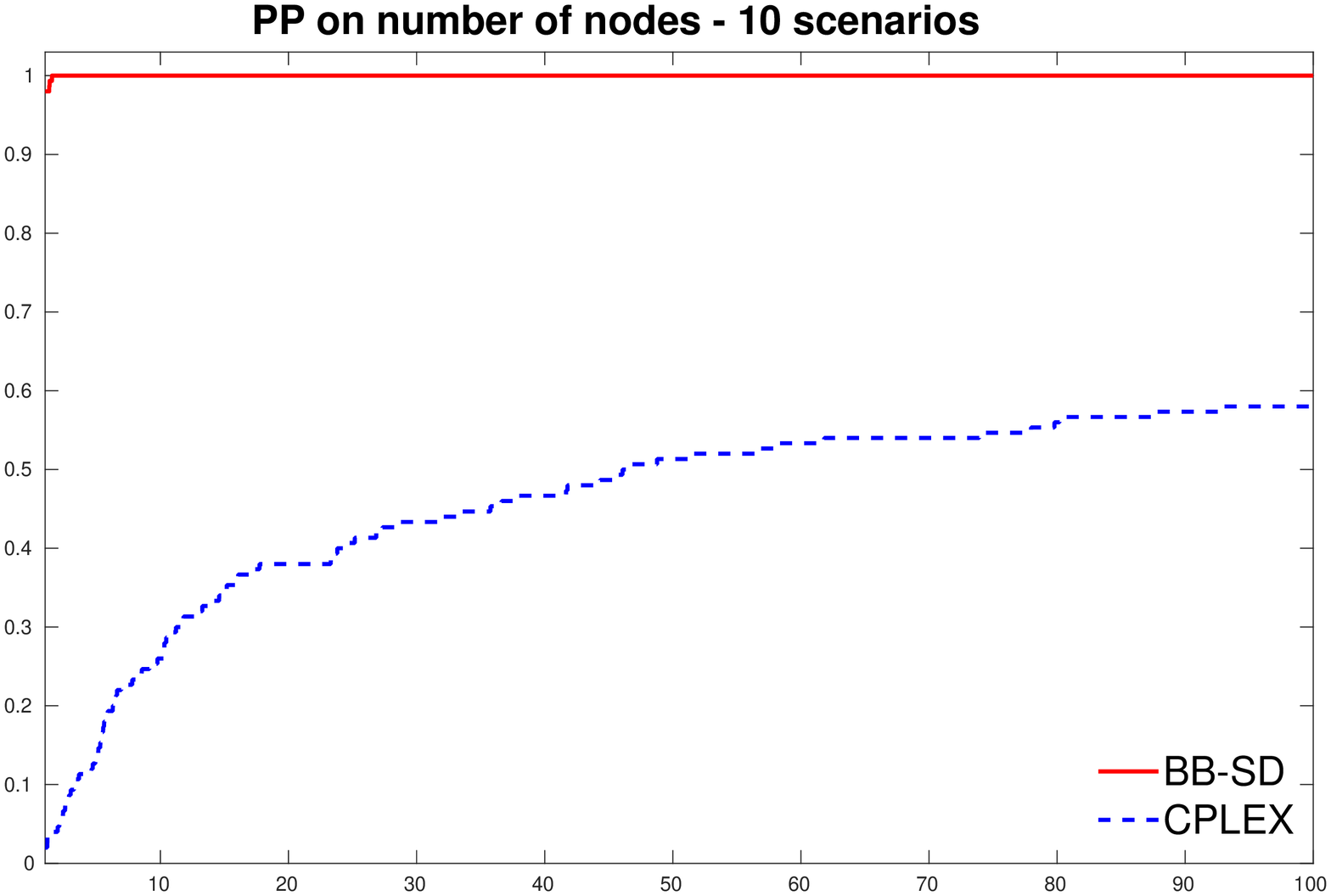}
    \includegraphics[trim={1.5cm 0.5cm 1cm 0},clip,width=0.49\textwidth]{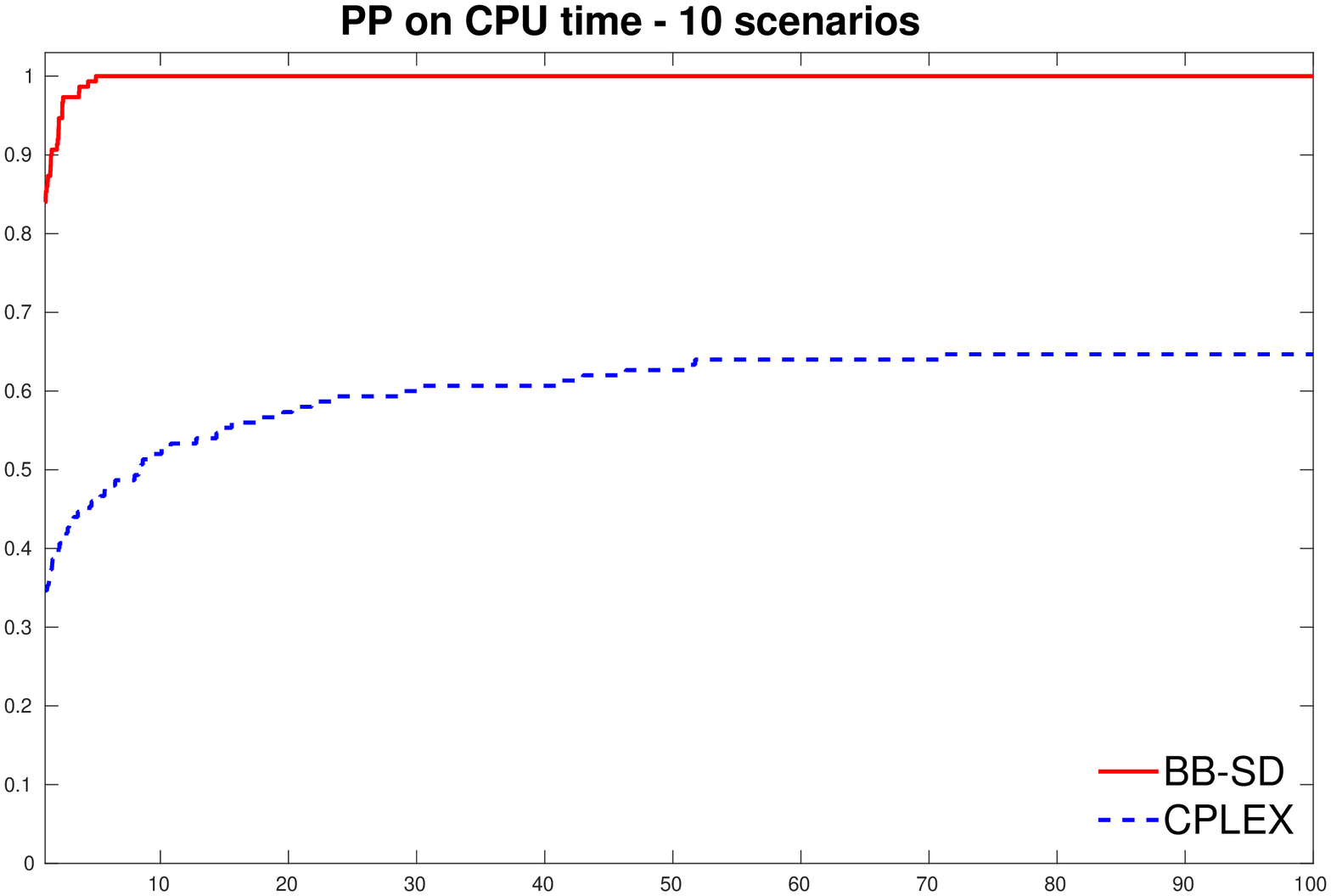}
    \\
    \includegraphics[trim={1.5cm 0.5cm 1cm 0},clip,width=0.49\textwidth]{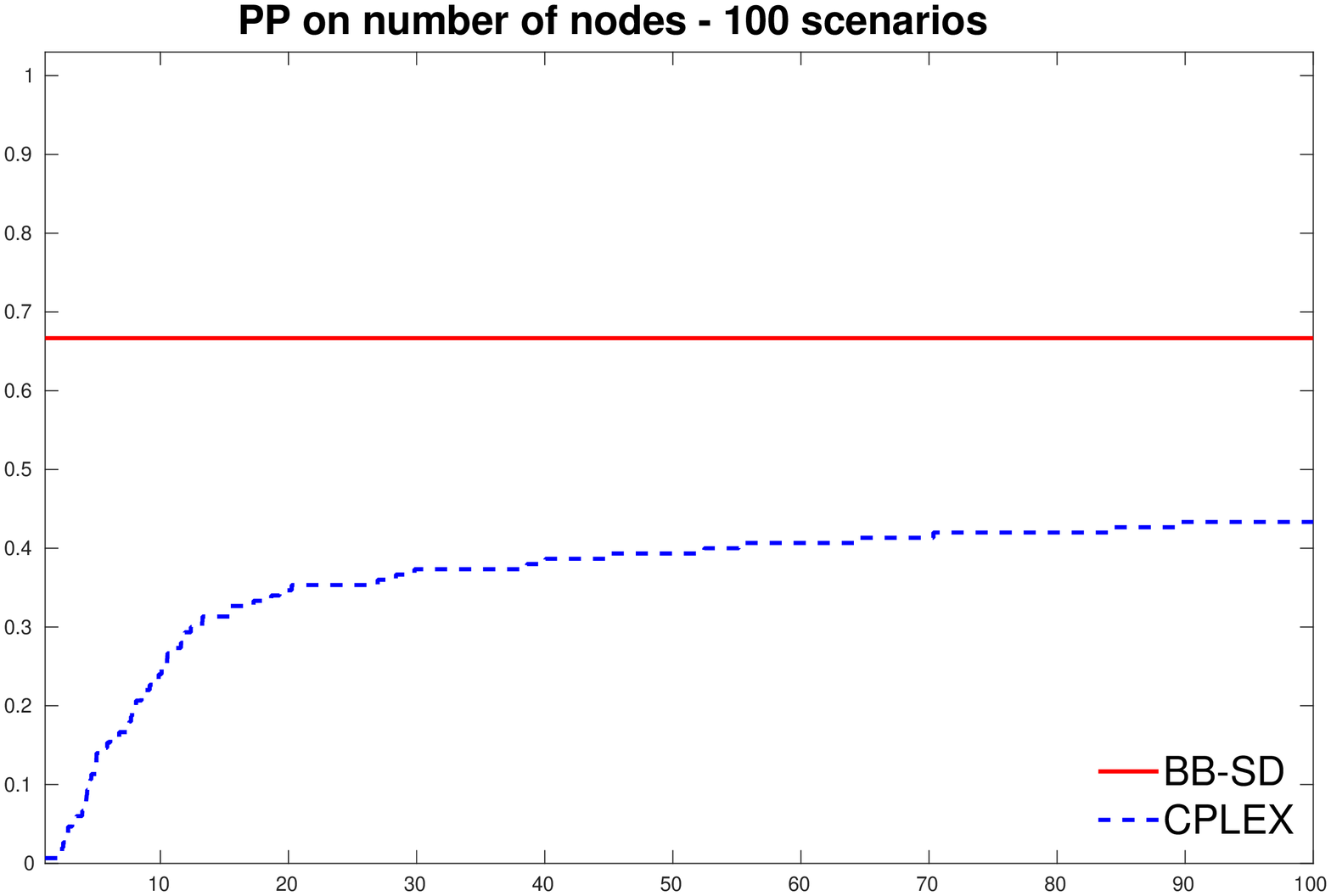}
    \includegraphics[trim={1.5cm 0.5cm 1cm 0},clip,width=0.49\textwidth]{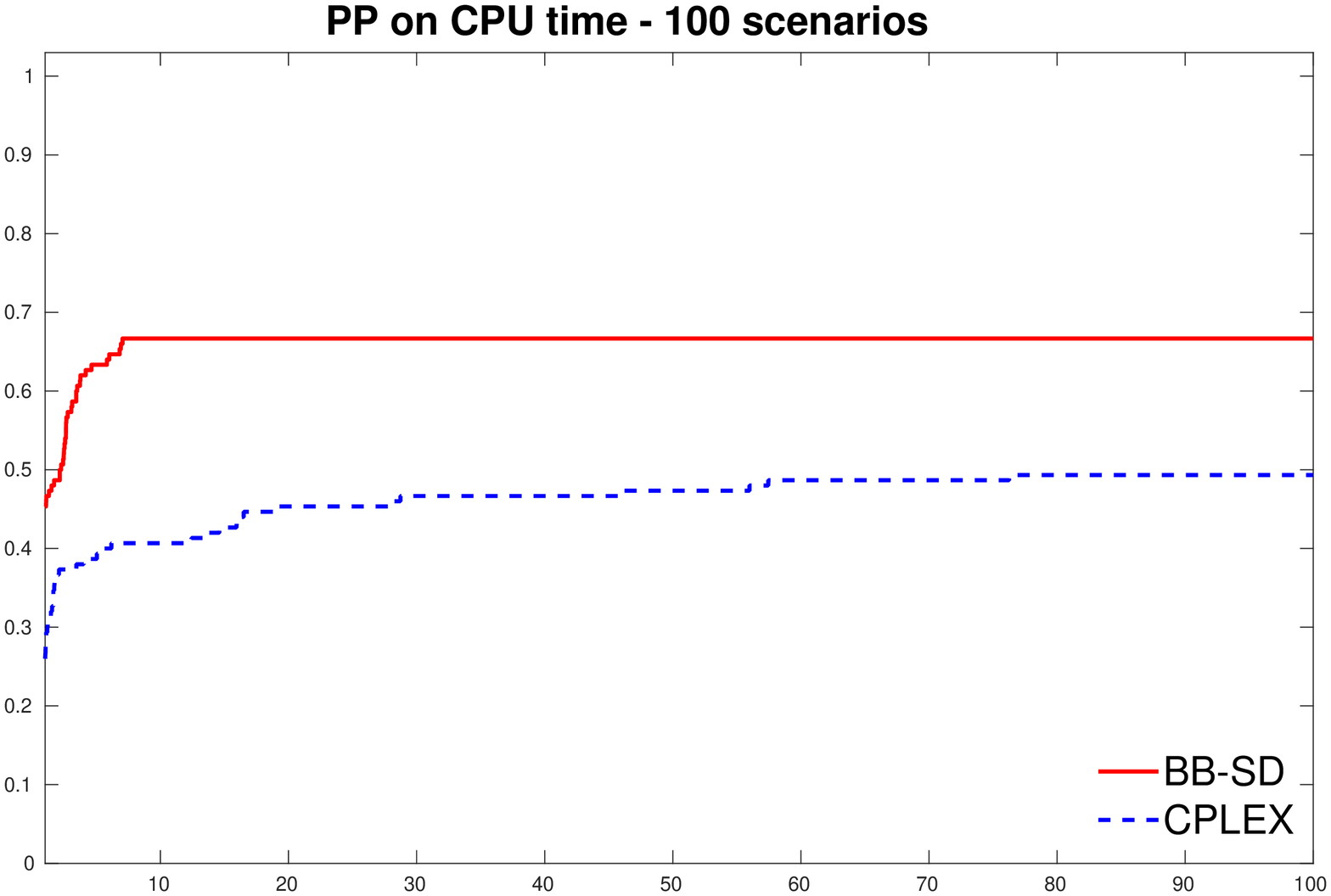}
    \\
    \includegraphics[trim={1.5cm 0.5cm 1cm 0},clip,width=0.49\textwidth]{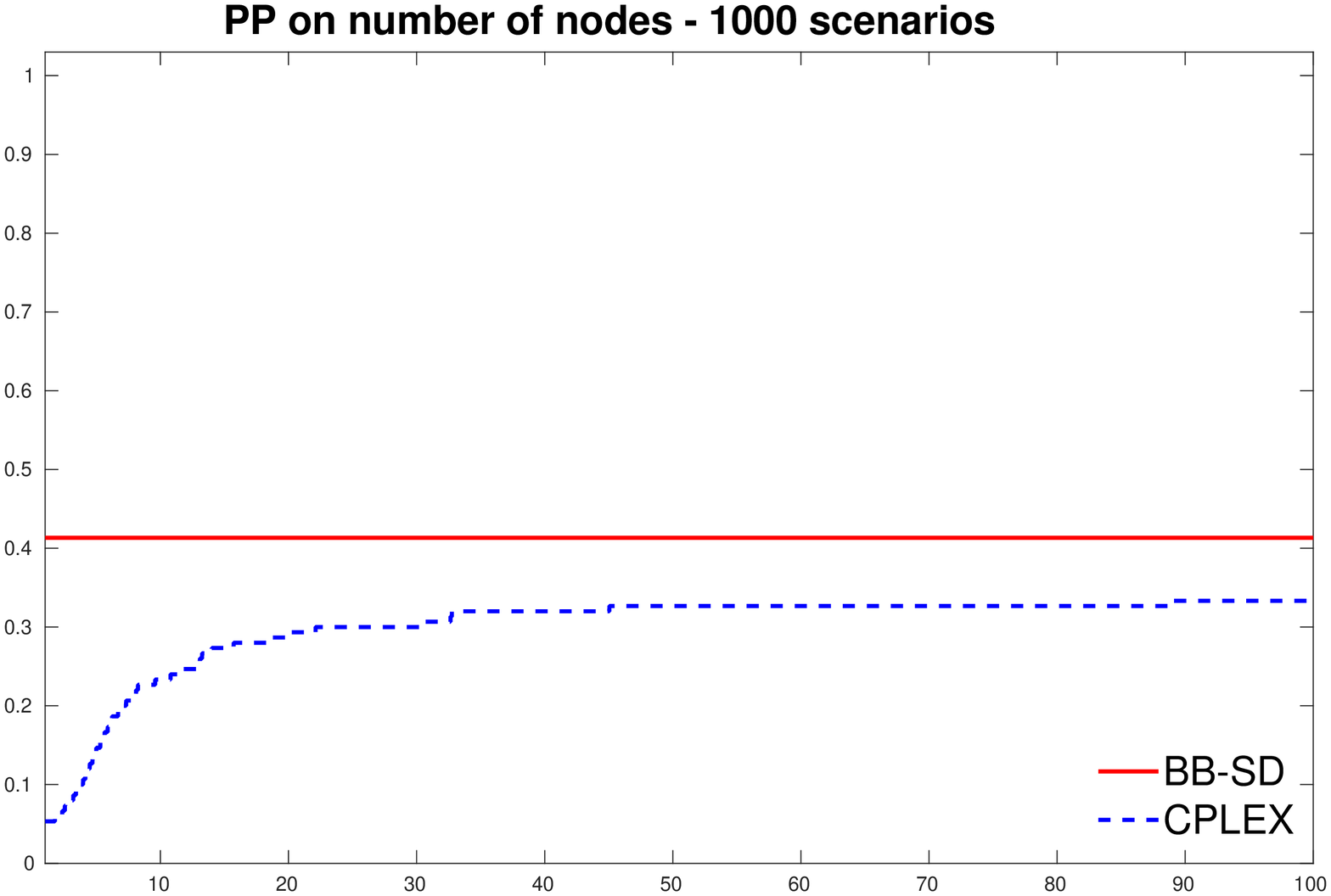}
    \includegraphics[trim={1.5cm 0.5cm 1cm 0},clip,width=0.49\textwidth]{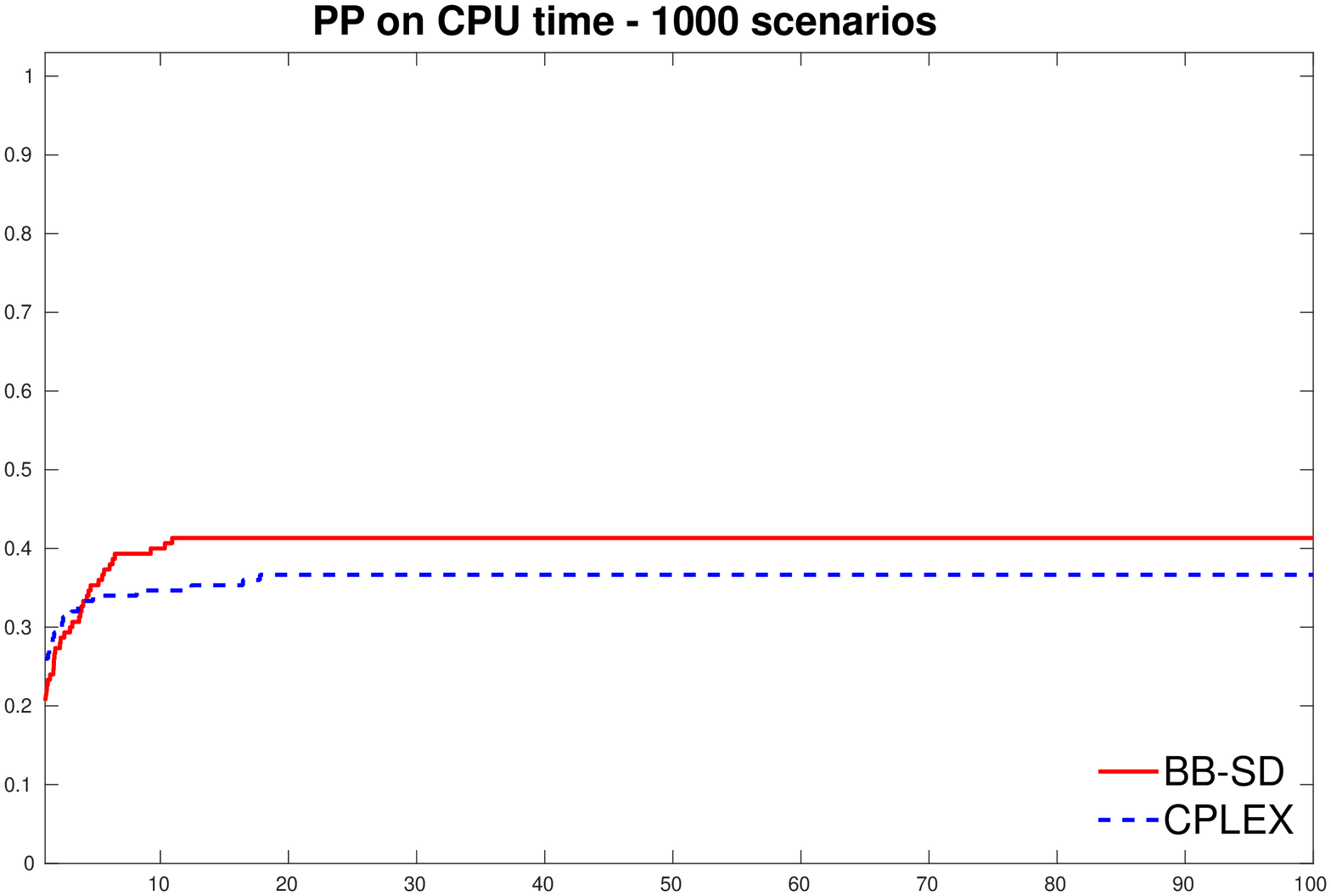}
    \caption{Comparison  between \texttt{BB-SD} and \texttt{CPLEX} on r-MST instances.}
    \label{fig:compCPLEX}
\end{figure}

\subsection{Traveling Salesman Problem}\label{sec:num-tsp}

Given an undirected, complete, and weighted graph $G =(N,E)$, the
Traveling Salesman problem consists in finding a path starting and
ending at a given vertex~$v\in N$ such that all vertices in the
graph are visited exactly once and the sum of the weights of its
constituent edges is minimized.  Our approach uses the following
formulation of the Traveling Salesman problem:
\begin{equation}\label{tsp}
  \begin{array}{rrcll}
    \min & \multicolumn{4}{l}{\max_{c\in U} c^\top x}\\[1.5ex]
    \textnormal{ s.t.} & \sum_{e\in \delta(i)}x_e & = & 2 & \forall i\in N \\[1ex]
    & \sum_{e\in \delta(X)} x_e & \ge & 2 & \forall \,\emptyset\neq X\subsetneq N\\[1ex]
    & x & \in & \{0,1\}^{|E|}
  \end{array}
\end{equation}
The number of inequalities is again exponential and for \texttt{CPLEX} we
use essentially the same separation algorithm as for the Spanning Tree problem;
see Section~\ref{sec:num-stp}.
 
For our tests, we consider 10 instances from the TSPLIB
library~\cite{tsplib}.  For each instance, we generate different
scenarios by adding to the nominal costs a random unit vector
multiplied by some coefficient.  This vector has non-negative
components, to avoid negative distances, and the coefficients are $1$,
$2$ and $3$, as for the r-MST case.  We again consider three different
numbers of scenarios, namely $10$, $100$, and $1000$, thus producing a
benchmark of $90$ instances in total.  As mentioned above, we realized
the linear oracle by using the solver \texttt{Concorde}. We used the
default version and solved each linear problem exactly. In particular,
in each linear oracle call an NP-hard problem is solved, so that the
time needed by LIN-O is now much larger than the time needed by SIM-O,
unlike in the MST case. Therefore, eliminating variables is not
effective: it would slightly reduce the time for the linear master
problems, while increasing the number of iterations and hence
increasing the time to solve the NP-hard oracles.  For this reason,
for our tests, we only consider \texttt{SD} where no dropping rule is
applied.

 In the following, we compare the performance of \texttt{SD} applied to
solve the continuous relaxation of the instances considered and the
performance of \texttt{CPLEX} at the root node.  We notice that
\texttt{CPLEX} minimizes the non-linear objective function~$\max_{c\in
  U} c^\top x$ over the subtour relaxation of the problem, while in
our formulation we implicitly optimize the same function over the
convex hull of feasible tours, thus obtaining a tighter lower
bound. However, our approach needs to solve NP-hard problems to
achieve this. It is thus not surprising that the computing time needed
by \texttt{SD} to solve the relaxation is often larger than the time
needed by \texttt{CPLEX} to solve its weaker relaxation.  However,
when requiring \texttt{CPLEX} to obtain the same stronger bound, the
required computational time increases significantly.  In
Table~\ref{tab:TSP}, we show the results for the TSP instances. For
every instance and every number of scenarios, we report the average
bound and computing time obtained by \texttt{SD} (SD root node) and by
\texttt{CPLEX} (CPLEX root node) to solve the continuous
relaxation. In the last column, we report the time needed by
\texttt{CPLEX} to obtain the same bound as the SD bound (CPLEX -- SD
bound).  The table shows that, on average, the SD bound is much
stronger than the subtour relaxation bound, but it is obtained in a
longer time.  Furthermore, the time needed by \texttt{CPLEX} to reach
the same bound as \texttt{SD} is often much larger than the SD
time and in one case the time limit of one hour was
reached.
\begin{table}\small
  \centering
  {\scalebox{.8}{
\begin{tabular}{|l|c|rr|rr|r|}
\hline
instance	&	\#sc	& \multicolumn{2}{c|}{SD root node}	&	\multicolumn{2}{c|}{CPLEX}	&	\multicolumn{1}{c|}{CPLEX}	\\
	&		& 	&	&      \multicolumn{2}{c|}{root node}	&      \multicolumn{1}{c|}{SD bound}	\\
	&		&	bound	&	time	&	bound	&	time	&	time	\\
	\hline
	\hline
brazil58	&	10	&	46031.5	&	2.09	&	41982.9	&	0.03	&	1.77	\\
	&	100	&	48086.6	&	12.35	&	43722.6	&	0.21	&	9.75	\\
	&	1000	&	49278.8	&	19.09	&	44810.6	&	2.68	&	256.33	\\
\hline
dantzig42	&	10	&	1158.0	&	1.59	&	1093.1	&	0.02	&	0.26	\\
	&	100	&	1203.6	&	0.97	&	1143.7	&	0.17	&	0.99	\\
	&	1000	&	1230.9	&	6.92	&	1166.3	&	3.50	&	15.33	\\
\hline
fri26	&	10	&	1622.8	&	0.45	&	1550.9	&	0.01	&	0.09	\\
	&	100	&	1691.7	&	1.32	&	1624.7	&	0.07	&	0.39	\\
	&	1000	&	1721.6	&	1.44	&	1663.8	&	0.64	&	2.75	\\
\hline
gr120	&	10	&	9801.5	&	15.96	&	9564.9	&	0.16	&	203.93	\\
	&	100	&	9977.8	&	68.51	&	9759.4	&	1.03	&	672.07	\\
	&	1000	&	10131.3	&	68.74	&	9886.2	&	18.34	&	$>3600.00$	\\
\hline
gr17	&	10	&	3911.2	&	0.28	&	3623.6	&	0.01	&	0.04	\\
	&	100	&	4053.9	&	0.57	&	3676.5	&	0.02	&	0.19	\\
	&	1000	&	4248.6	&	0.73	&	3925.4	&	0.16	&	1.94	\\
\hline
gr21	&	10	&	4928.4	&	0.22	&	4903.8	&	0.01	&	0.01	\\
	&	100	&	5138.5	&	0.22	&	5104.6	&	0.03	&	0.15	\\
	&	1000	&	5301.1	&	0.23	&	5277.2	&	0.24	&	0.76	\\
\hline
gr24	&	10	&	2202.9	&	0.30	&	2153.9	&	0.01	&	0.04	\\
	&	100	&	2272.6	&	0.55	&	2251.5	&	0.04	&	0.16	\\
	&	1000	&	2359.8	&	1.76	&	2318.3	&	0.37	&	1.52	\\
\hline
gr48	&	10	&	7642.8	&	0.92	&	7417.6	&	0.02	&	0.43	\\
	&	100	&	7907.4	&	6.38	&	7668.2	&	0.14	&	3.10	\\
	&	1000	&	8034.1	&	8.38	&	7789.8	&	1.57	&	31.01	\\
\hline
hk48	&	10	&	18545.0	&	0.55	&	17895.9	&	0.02	&	0.40	\\
	&	100	&	18889.0	&	1.50	&	18377.3	&	0.14	&	2.18	\\
	&	1000	&	19190.6	&	6.74	&	18854.7	&	1.63	&	21.43	\\
\hline
swiss42	&	10	&	2051.0	&	1.71	&	1970.7	&	0.03	&	0.50	\\
	&	100	&	2128.4	&	1.41	&	2064.1	&	0.16	&	1.25	\\
	&	1000	&	2185.3	&	9.20	&	2117.2	&	2.13	&	19.31	\\
	\hline
\end{tabular}
}}
\caption{Average results for r-TSP, continuous relaxations.}
\label{tab:TSP}
\end{table}

\section{Conclusion}\label{section:conclusion}

We presented an algorithm for the exact solution of strictly robust
counterparts of combinatorial optimization problems, entirely based on
a linear optimization oracle for the underlying problem. Concentrating
on the discrete scenario case, our experimental evaluation shows that
the approach is competitive both in case the underlying problem is
very easy to solve, as in the MST case, and in case it is a hard
problem, as in the TSP case. In particular, in the latter case, we
have seen that solving the underlying problem to optimality can be
beneficial even when it is NP-hard: in the same amount of time, our
approach produces much better dual bounds than CPLEX based on a
linearized IP formulation of the problem, using the standard subtour
formulation.

We emphasize again that our approach is not restricted to the case of
discrete uncertainty. However, the oracle SIM-O must be adapted when
considering other classes of uncertainty sets. In case of ellipsoidal
uncertainty, SIM-O turns out to be a second-order cone program. As
mentioned above, since~$f$ is a smooth function in this case,
cycling is not possible even if the most aggressive dropping
rule~\eqref{eq:elrule} is used. For the case of polyhedral
uncertainty, SIM-O can again be realized as a linear program, and the
statement of Theorem~\ref{theorem_perturb} holds analogously.

The investigation of other generalizations of our approach is left as future work. In particular, 
it may be interesting to extend it to more general classes of uncertain objective functions, e.g., of the 
form~$c^\top g(x)$, where a convex function~$g\colon\R^n\to\R^m$ is given and the coefficients~$c\in\R^m_+$ are uncertain. 
In this case, the function~$f$ decribing the worst case over all scenarios is still a convex function, and it suffices to 
adapt the oracle SIM-O.


\bibliographystyle{plainnat}
\bibliography{references}

\end{document}